\documentclass[final]{siamart1116}

\usepackage[textsize=tiny]{todonotes}

\usepackage{amssymb}
\usepackage{amsfonts}
\usepackage{mathtools}

\usepackage{subcaption}

\usepackage{chngcntr}
\usepackage{booktabs}

\usepackage{algorithm}
\usepackage{algorithmic}

\newtheorem{remark}[theorem]{Remark}
\newtheorem{open lemma}[theorem]{Open Lemma}

\counterwithin{figure}{section}
\counterwithin{table}{section}
\counterwithin{equation}{section}
\counterwithin{theorem}{section}

\newcommand{\algorithmicbreak}{\textbf{break}}
\newcommand{\BREAK}{\STATE \algorithmicbreak}

\newcommand{\mc}{\mathcal}

\newcommand{\trunc}[2]{\mc T\left(#1,#2\right)}
\newcommand{\bs}{\boldsymbol}

\newcommand{\J}{\mathcal{J}}

\newcommand{\R}{\mathbb R}
\newcommand{\N}{\mathbb N}

\newcommand{\cP}{\mathcal{P}}
\newcommand{\eps}{\varepsilon_1}
\newcommand{\epss}{\varepsilon_2}
\newcommand{\pp}{\tilde{p}}

\newcommand{\qq}{\tilde{q}}

\newcommand{\gt}{\tilde{g}}
\newcommand{\CG}{\text{CG}}

\newcommand\ds[1]{\displaystyle{#1}}

\DeclarePairedDelimiter\innerp{\langle}{\rangle}

\DeclareMathOperator*{\argmin}{arg\,\min}
\DeclareMathOperator{\supp}{supp}

\author{Mazen Ali\footnotemark[1]
	\and
	Karsten Urban%
	\thanks{Ulm University, Inst.\ f.\ Numerical Mathematics,
	Helmholtzstr.\ 20, D-89081 Ulm, Germany,
	\{mazen.ali,karsten.urban\}@uni-ulm.de}
	}

\title{HT-AWGM: A Hierarchical Tucker--Adaptive Wavelet
Galerkin Method for High Dimensional Elliptic Problems.}

\begin{document}

\maketitle
\begin{abstract}
This paper is concerned with the construction, analysis and realization of
a numerical method to approximate the solution of high dimensional elliptic
partial differential equations.
We propose a new combination of an Adaptive Wavelet Galerkin Method (AWGM) and
the well-known Hierarchical Tensor (HT) format.
The arising HT-AWGM is adaptive both in the wavelet representation of the low
dimensional factors and in the tensor rank of the HT representation.

The point of departure is an adaptive wavelet method for the HT format
using approximate Richardson iterations from \cite{DB} and an AWGM method
as described in \cite{GHS}.
HT-AWGM performs a sequence of Galerkin
solves based upon a truncated preconditioned conjugate gradient (PCG) algorithm
from \cite{ToblerPhd} in combination with a tensor-based
preconditioner from \cite{MarkusSobolev}.

Our analysis starts by showing convergence of the truncated
conjugate gradient method.
The next step is to add routines realizing the adaptive refinement.
The resulting HT-AWGM is analyzed concerning convergence and complexity.
We show that the performance of the 
scheme asymptotically depends only on the desired tolerance
with convergence rates depending on the Besov regularity of low dimensional
quantities and the low rank tensor structure of the solution.
The complexity in the ranks is algebraic with powers of four stemming
from the complexity of the tensor truncation. Numerical experiments
show the quantitative performance.
\end{abstract}

\begin{keywords}
    High Dimensional, Hierarchical Tucker, Low-Rank Tensor Methods,
    Adaptive Wavelet Galerkin Methods, Partial Differential Equations
\end{keywords}

\begin{AMS}
    65N99
\end{AMS}

\section{Introduction}

The increase of available computational power made a variety of complex problems accessible for computer-based simulations. However, the complexity of problems has increased even faster, so that several `real-world' problems will be out of reach even with computers of the next generations. One class of such challenging problems arises from high-dimensional models suffering from the \emph{curse of dimensionality}. This shows the ultimate need to construct and analyze sophisticated numerical methods. 

This paper is concerned with high-dimensional systems of elliptic partial differential equations (PDEs). Examples include chemical reactions, financial derivatives, equations depending on a large number of parameters (e.g.\ material properties) or a large number of independent variables. In general terms, we consider an operator problem $Au=f$, where $A:\mc X\to \mc X'$ is elliptic\footnote{We assume that $\mc X\hookrightarrow \mc H\hookrightarrow \mc X'$ is a Gelfand triple with a pivot Hilbert space $\mc H$ and $\mc X'$ is the dual space of $\mc X$
induced by $\mc H$.}, $f\in \mc X'$ is given and $u\in\mc X$ is the desired solution, which we aim to approximate in a possible `sparse' manner.

Of course, this issue also depends on the specific notion of sparsity, which itself is typically adapted to the problem. In the context of adaptive methods (think of adaptive finite element or wavelet methods), the sparsity benchmark is a \emph{Best $N$-term} approximation, i.e., a possibly optimal approximation to $u\in \mc X$ using $N\in\N$ degrees of freedom. In particular for high-dimensional problems, one tries to approximate $u$ in terms of \emph{low rank} tensor format approximations. We will combine these two notions to be explained next.

\subsubsection*{Best $N$-term Approximation}
Given a dictionary (basis, frame) $\Psi:=\{\psi_\lambda:\lambda\in\mc J\}\subset\mc X$, where the index set $\mc J$ is typically of infinite cardinality, one seeks an approximate expansion of $u$ in $\Psi$. A best $N$-term approximation is of the form $u\approx u_N:=\sum_{\lambda\in\Lambda}c_\lambda\psi_\lambda$, $c_\lambda\in\R$ and $\Lambda\subset\mc J$ is of cardinality $N\in\N$, i.e., $|\Lambda|=N$.The goal of an optimal approximation can also
be expressed by determining the minimal number of terms $N(\varepsilon)$ required to achieve a certain accuracy $\varepsilon>0$: $\| u-u_{N(\varepsilon)}\|_{\mc X}\le\varepsilon$.

It is known that the optimal speed of convergence of such approximations entirely depends on the properties of the solution $u$
and the chosen basis. In fact, there is an intimate connection between decay of the error of the best $N$-term approximation  and the Besov regularity of $u$, see \cite{DeVore}. An approximation scheme (or algorithm) is called \emph{quasi-optimal} if it realizes (asymptotically) the same rate as the $N$-term approximation. Known quasi-optimal methods are adaptive in the sense that approximations are constructed in nonlinear manifolds rather than in linear subspaces. 

For Adaptive Finite Element Methods (AFEM, \cite{AFEM}) and Adaptive Wavelet Methods (AWM, e.g.\ \cite{CDD1,CDD2,GHS}) there are quasi-optimal algorithms known, in particular for elliptic problems. 

\subsubsection*{Low-Rank Tensor Methods}
For high-dimensional problems ($d\gg 1$), it is well-known that most algorithms scale
exponentially in the dimension and are thus intractable: they suffer from the curse of dimensionality. If the operator $A$ has a tensor structure (or can at least be well-approximated by such), one can try to find an efficient separable approximation
\begin{align}\label{eq:seprep}
    u\approx\sum_{i=1}^r\bigotimes_{j=1}^dv_j^i,
\end{align}
where $r$ is referred to as the \emph{rank} and $v_1\otimes\cdots\otimes v_d(x):= v_1(x_1)\cdots v_d(x_d)$ for $x=(x_1,\ldots ,x_d)\in\R^d$ is a tensor product. Hence, if the rank $r$ is small even for large $d$, one can try to approximate the univariate factors $v_j^i:\R\rightarrow\R$ separately resulting in a tractable algorithm. 

A major breakthrough in this area was the development of tensor formats that in fact realized such approximations. We mention \emph{the hierarchical Tucker (HT) format} \cite{HT}, the \emph{tensor train format} \cite{TT} and refer to \cite{Hackbusch} for a general overview. Nowadays, there is a whole variety of algorithms that have been
developed in these formats, both iterative solvers \cite{BallaniGrasedyck, Khoromskij, KS, KT} (using basic arithmetic operations on tensors and truncations to control the rank)  and direct methods \cite{MasterEq, ALS, QuanticsTT, OD}, which work within the tensor structure itself. For a survey on tensor methods for solving high-dimensional PDEs we refer to \cite{TensorOverview}.

\subsubsection*{HTucker-Adaptive Wavelet Galerkin Method (HT-AWGM)}
In this paper, we consider a combination of best $N$-term and low rank approximations in order to obtain a convergent algorithm that is optimal both w.r.t.\ $N$ and the tensor rank $r$. To this end, we use appropriate wavelet bases $\Psi$, i.e., the factors in \cref{eq:seprep} are approximated by sparse wavelet expansions
\begin{align*}
    v_j^i=\sum_{\lambda\in\Lambda_j^i}
    c_\lambda^{i,j}\psi_\lambda^j,\quad c_\lambda^{i,j}\in\R.
\end{align*}

To the best of our knowledge, the first such approximation was constructed in
\cite{DB}, where inexact Richardson iterations from \cite{CDD2} were combined
with the HT format from \cite{HT}.
In \cite{ST}, the authors considered soft threshholding techniques for the rank
reduction.
Though convergence and complexity estimates
were provided, it is still unclear what is the correct notion of optimality
for high-dimensional problems.

The goal of this paper is to extend the AWGM method to the
high dimensional setting using the HT format -- resulting in an \emph{HT-AWGM}.
In particular, we aim at providing the corresponding convergence analysis.
A core ingredient of AWGM is the fact that wavelet bases can be used to
rewrite the operator equation $Au=f$ equivalently into an equation
$\bs A\bs u=\bs f$ in sequence spaces, where $\bs A$ is boundedly invertible.
The backbone of that is optimal wavelet preconditioning.
Hence, a tensor-based wavelet preconditioner is needed.
Luckily, in \cite{MarkusSobolev} the problem of separable preconditioning
was addressed and the algorithm from \cite{DB} was extended
to the elliptic case.

\subsubsection*{Organization of the Paper}
The remainder of this paper is organized as follows.
In Section \ref{sec:prelim}, we collect all required preliminaries.
As a core ingredient for the new HT-AWGM,
we use a truncated PCG algorithm from \cite[Algorithm 9]{ToblerPhd}
and analyze its convergence in Section \ref{sec:pdescent}.
The convergence and complexity analysis of the full HT-AWGM is described in
Section \ref{sec:awgm}.
We show numerical results in Section \ref{sec:numexp}. We indicate the
potential and remaining issues of the method.


\section{Preliminaries}\label{sec:prelim}

We start by briefly reviewing some basic facts on adaptive wavelet methods, low rank tensor formats and the preconditioning problem arising in connection with tensor spaces.

\subsection{(Quasi-)optimal Approximations}\label{sec:wavs}
For the remainder of this work we use the shorthand notation
\begin{align*}
    A\lesssim B,
\end{align*}
to indicate there exists a constant $C>0$ independent of $A$ and $B$ such
that $A\leq CB$. The notation $A\gtrsim B$ is defined analogously.

The introduction mainly follows \cite{Stevenson}. We seek the solution of the operator equation
\begin{align}
    A u=f,\quad A:\mc X\rightarrow\mc X',\quad u\in\mc X,\quad f\in\mc X',
    \label{eq:opeq}
\end{align}
where $A$ is a linear boundedly invertible operator and $\mc X$ is a
separable Hilbert Space. Given a Riesz basis $\Psi:=\left\{\psi_\lambda:\lambda\in\J\right\}$,
e.g., a wavelet basis, and the corresponding boundedly invertible analysis and synthesis
operators
\begin{align*}
    \mc F:\mc X'\rightarrow\ell_2(\J),\quad f\mapsto\{f(\psi_\lambda)\}_\lambda,
    &\qquad&
    \mc F':\ell_2(\J)\rightarrow\mc X,\quad \{c_\lambda\}_\lambda\mapsto
    \sum_{\lambda\in\J}c_\lambda\psi_\lambda,
\end{align*}
we can reformulate \cref{eq:opeq} equivalently as a discrete infinite dimensional linear
system
\begin{align}
    \bs A\bs u=\bs f,\quad \bs A:\ell_2(\J)\rightarrow\ell_2(\J),
    \quad \bs u,
    \bs f\in\ell_2(\J),\label{eq:opeqdisc}
\end{align}
with $\bs A:=\mc F A\mc F'$, $\bs u:=\mc F\mc R u$ and $\bs f:=\mc F f$, where $\mc R:\mc X\rightarrow\mc X'$ is the Riesz isomorphism. The operator $\bs A$ inherits the properties of its continuous counterpart $A$ and is in particular boundedly invertible as well.

Next, we introduce the notation for the Galerkin problem. Let
$\Lambda\subset\J$ be some finite index subset. We introduce
the restriction operator $R_{\Lambda}:\ell_2(\J)\rightarrow
\ell_2(\Lambda)$, which simply drops all entries outside
$\Lambda$. Likewise the extension operator
$E_{\Lambda}:\ell_2(\Lambda)\rightarrow
\ell_2(\J)$ pads all entries outside $\Lambda$ with zeros. We will
sometimes employ the notation $\bs A_{\Lambda}:=R_{\Lambda}
\bs A E_{\Lambda}$ to denote the discretized wavelet operator.

The benchmark for optimal approximations is the \emph{best $N$-term
approximation}
\begin{align*}
    u_N:=\argmin\Big\{\|u-v\|_{\mc X}:v\in\mc X,\;
    v=\sum_{\lambda\in \Lambda\subset\J}v_\lambda\psi_\lambda,\;
    \#\Lambda\leq N\Big\},
\end{align*}
or, equivalently, in $\ell_2(\J)$
\begin{align*}
    \bs u_N:=\argmin\left\{\|\bs u-\bs v\|_{\ell_2}:\bs v\in\ell_2(\J),\;
    \#\supp(\bs v)\leq N\right\},
\end{align*}
where $\supp(\bs v)$ denotes those wavelet indices $\lambda\in\J$, for which $\bs v_\lambda\neq 0$.
Note that, as opposed to linear approximation techniques, we seek an
approximation
in an $N$-dimensional nonlinear manifold.
The approximation class of all best $N$-term approximations converging with
rate $s$ is known as
\begin{align}\label{eq:bestN}
    \mc A_s:=\Big\{\bs u\in\ell_2(\J):
    \|\bs u\|_{\mc A_s}:=\sup_{\varepsilon>0}\varepsilon
    [\min\{N\in\N_0 :\|\bs u-\bs u_N\|_{\ell_2}\leq\varepsilon\}]^s<\infty
    \Big\}.
\end{align}
It is known that such approximation spaces are interpolation spaces between
$L_p$ and certain Besov spaces, which establishes a direct link between regularity and
approximation classes, see also \cite{DeVore} for more details.

An adaptive wavelet method is called \emph{(quasi-)optimal} whenever it
produces for
$\bs u\in\mc A_s$ an approximation $\bs v$ to $\bs u$ with
$\|\bs u-\bs v\|_{\ell_2}\leq\varepsilon$, such that $\#\supp(\bs v)\lesssim
\varepsilon^{-1/s}\|\bs u\|^{1/s}_{\mc A_s}$ and the number of operators is
bounded by a multiple of the same quantity. In other words, given that
$\bs u$ is in
a certain approximation class, an optimal adaptive
method achieves the best possible asymptotic rate of convergence
in linear computational
complexity of the output size.

There are two classical approaches to implementing such an optimal adaptive
wavelet method (see \cite{CDD1,CDD2,GHS}). The
first\footnote{Chronologically, however, the second.}
applies an \emph{inexact iteration method} such as the Richardson iteration, to the
bi-infinite discrete system in \cref{eq:opeqdisc}. The second one, in the
spirit of adaptive FEM methods, produces a sequence $\Lambda^{(0)}\rightarrow
\Lambda^{(1)}\rightarrow\cdots$ of finite index sets and solves the finite Galerkin
problem on these sets, yielding a sequence of solutions $\bs u^{(0)}
\rightarrow\bs u^{(1)}\rightarrow\cdots$, following the paradigm
solve $\rightarrow$ estimate $\rightarrow$ mark $\rightarrow$ refine.
The latter one is referred to as an \emph{adaptive
wavelet Galerkin method (AWGM)}, which is the focus of this paper.

There are three basic routines necessary for an efficient realization of an
AWGM: (1) approximate residual evaluation (Estimate),
(2) approximate Galerkin solver\\
(solve) and (3) bulk chasing (mark and refine).
We do not discuss these routines in detail here, but refer to the literature. In order to control the number of active variables (number of selected wavelets), one often uses a coarsening step in order to remover `unnecessary' coefficients. This is done by a routine called \textbf{COARSE}, which we detail for later use:  
For a given finitely supported $\bs v$ such routine is assumed to produce an approximation $\bs v_\varepsilon$ such that
\begin{align*}
    \#\supp(\bs v_\varepsilon)\lesssim\min\left\{N:\|\bs v-\bs w\|_{\ell_2}
    \leq\varepsilon,\;\bs w\in\ell_2(\J),\;\#\supp(\bs w)\leq N\right\}.
\end{align*}
A straightforward realization would involve sorting -- with log linear
complexity. To achieve linear complexity,
exact sorting can be replaced by an approximate bin sorting which
satisfies the above estimate. Again, we refer to the literature.

Note that both methods require that $A$, or, equivalently, $\bs A$ is
symmetric positive definite. Otherwise a similar analysis applies to
the normal equations with $\bs A^T\bs A$. However, the additional application
of $\bs A^T$ hampers numerical performance and convergence estimates depend
on $\kappa(\bs A)^2$ rather than on $\kappa(\bs A)$. The penalty for applying
$\bs A^T$ is even more severe in the high-dimensional case
due to the increase in ranks.

\subsection{Tensor Formats}\label{sec:tensors}
We briefly review some of the basics of tensor formats, see e.g.\ \cite{Hackbusch}. 
In this paper, we view \emph{tensors} as algebraical or topological objects rather than tensor fields as geometrical objects\footnote{By the \emph{universality property} an equivalence between
the two concepts can be established, \cite{Kreyszig}.}.
A \emph{tensor of order $d$} is an element of a tensor space
$\mc V:=\otimes_{j=1}^dV_j$, where $V_j$ are some vector spaces. We consider
\emph{topological} tensor spaces, i.e., $\mc V$ is Banach space with some norm
$\|\cdot\|_{\mc V}$. Typically, $V_j$ are themselves Banach spaces and the
norm on $\mc V$ is induced by the norms on $V_j$. The \emph{tensor product}
$\otimes:V_1\times\cdots\times V_d\rightarrow V_1\otimes\cdots\otimes V_d$
is the unique multilinear mapping factoring any other multilinear mapping
$\varphi:V_1\times\cdots\times V_d\rightarrow W$ into a linear mapping
$f:V_1\otimes\cdots\otimes V_d\rightarrow W$ such that $\varphi=f\circ\otimes$,
where $V_j$ and $W$ are some vector spaces.
If the tensor product $\otimes:\ds{\bigtimes_{j=1}^d}(V_j,\|\cdot\|_{V_j})
\rightarrow(\mc V,\|\cdot\|_{\mc V})$
is continuous, any element $u\in\mc V$ can be written as
\begin{align}\label{eq:cp}
    u=\sum_{k=1}^r\bigotimes_{j=1}^dv_j^k,
\end{align}
with $r\leq\infty$. The representation in \cref{eq:cp} is referred to as the
\emph{$r$-term representation} or \emph{CP format} (canonical polyadic
decomposition). The smallest possible $r$ in this representation is called the
\emph{tensor rank} and we will denote it by
$$
r(u)\in\N_0\cup\{\infty\}, 
$$
whenever it is clear that $u$ is to be interpreted in the $r$-term format.
Though the representation \cref{eq:cp} would be a cheap way to
store $u$, the approximation problem in the said format is ill posed, the
reason being already apparent from \cref{eq:cp}, namely possible cancellations.
A format which is better suited for approximation is the \emph{Tucker format}
\begin{align*}
    u=\sum_{i_1=1}^{r_1}\ldots\sum_{i_d=1}^{r_d}a_{i_1,\ldots,i_d}
    \bigotimes_{j=1}^dU_j^{i_j}=:Ua,
\end{align*}
with
\begin{align*}
    U:=\bigotimes_{j=1}^dU_j,\quad
    &U_j:=[U_j^1,\ldots,U_j^{r_j}],\quad
    &a:=[a_{i_1,\ldots,i_d}],\, 1\leq i_j\leq r_j,\;1\leq j\leq d,
\end{align*}
where the $U_j$'s are referred to as \emph{frames} and $a$ as \emph{core tensor}.
One can apply techniques from (multi)linear algebra in combination with
matricizations to build a well conditioned, even orthonormal basis $U$.
Unfortunately, the storage cost of the core tensor $a$ grows exponentially in
$d$.

The \emph{hierarchical Tucker (HT) format} combines both the advantages of
stable approximation of the Tucker format with the sparse representation of
the $r$-term format by further decomposing the core tensor. For a general
multi-index $\alpha\subset\{1,\ldots, d\}$, we can define the tensor product
vector space
\begin{align*}
    V_\alpha:=\bigotimes_{j\in\alpha}V_j.
\end{align*}
The idea behind
HT can be illustrated by the following simple observation: An element
$u\in\mc V$ can be also seen as an element of $u\in V_{\alpha}\otimes
V_{\bar\alpha}$ with $\alpha,\bar\alpha\subset\{1,\ldots,d\}$ with $\bar\alpha$ being
the complement of $\alpha$. Note, that the rank $r(u)$ may change
if we reinterpret $u$.
Applying this idea recursively, we start with a Tucker decomposition of
$u\in V_\alpha\otimes V_{\bar\alpha}$. We then further decompose the bases
$U_\alpha$ and $U_{\bar\alpha}$ of $V_{\alpha}$ and $V_{\bar\alpha}$ respectively,
until we reach the singeltons $\alpha=\{j\}$.
We denote the ranks of this hierarchical representation by
$r(u)=(r(u)_\alpha)_{\alpha\in T}$ with the max norm
$|r(u)|_\infty$ defined in an obvious way, where $T$ is the HT tree
structure. In contrast to the Tucker format,
which requires the storage of an order $d$ tensor, the HT format stores
several order 3 tensors\footnote{Due to the binary decomposition
$\alpha=\alpha_L\cup\alpha_R$,
each transfer tensor has 2 indices related to the child nodes
$\alpha_L$, $\alpha_R$ and one index related to the parent node $\alpha$.}.
However, note that in the worst case $r(u)$ can still behave exponentially w.r.t.\ $d$.
Nonetheless, it is known that the asymptotic behavior of the storage requirements of HT are
not worse than that of the $r$-term format and the performance of HT in
practice has proven its merit. A rigorous answer to the question as to when
and why functions exhibit good approximation properties in tensor tree formats
remains a challenging and interesting problem.

As in the case for best $N$-term approximations in \cref{eq:bestN}, we require
a benchmark to assess the quality of the ranks of approximation. For
this purpose we use the benchmark introduced in \cite{DB}, similar to
\cref{eq:bestN}. We use the notation $u\in\mc H_N$ to denote that
$u$ is representable in an HT format with $|r(u)|_\infty\leq N$.
Given a positive, strictly increasing growth sequence,
$\gamma:=(\gamma(n))_{n\in\N_0}$ with $\gamma(0)=1$, define an approximation class
as
\begin{align*}
    \mc A(\gamma):=\left\{v\in\mc V:|v|_{\mc A(\gamma)}
    :=\sup_{N\in\N_0}
    \gamma(N)\inf_{w\in\mc H_N}\|v-w\|_{\mc V}<\infty\right\},
\end{align*}
with norm $\|v\|_{\mc A(\gamma)}=\|v\|_{\mc V}+|v|_{\mc A(\gamma)}$.
It is known from, e.g., \cite{SchneiderUsch} that the best approximation error for
a function with Sobolev
smoothness $s$ behaves in the worst case like
\begin{align*}
    \max_{\alpha\in T\setminus\{1,\ldots,d\}}r_\alpha^{-s
    \max\{1/|\alpha|, 1/(d-|\alpha|)\}}.
\end{align*}

One of the most important operations on tensors is truncation. It lies in the
heart of all iterative tensor algorithms that rely on truncation to keep ranks
low. For a given algebraic tensor $u\in\mc V$, we seek an approximation $v\in\mc V$
with $r(v)_\alpha\leq r_\alpha\leq r(u)_\alpha$ for some fixed $r_\alpha$
and all $\alpha\subset\{1,\ldots,d\}$.
In practice, this can be done by applying singular value decompositions (SVD) to
matricizations $\mc M_\alpha(u)\in V_\alpha\otimes V_{\bar\alpha}$, a method referred to as
\emph{higher order singular value decomposition} (HOSVD). Unlike the standard SVD, the HOSVD
provides one only with a quasi-best approximation in the sense
\begin{align}\label{eq:hosvd}
    \|u-v_{\text{HOSVD}}\|_{\mc V}\leq\sqrt{\sum_\alpha\sum_{i\geq r_\alpha+1}
    (\sigma_i^\alpha)^2}\leq\sqrt{2d-3}\inf_{\substack{v\in\mc V,\\
    r(v)\leq r}}\|u-v\|_{\mc V},
\end{align}
where $r=(r_\alpha)_\alpha$ is some integer vector and $\sigma_i^\alpha$ are
the corresponding singular values of the $\alpha$ matricization. We will
denote the (nonlinear) operator that produces an HOSVD of $u$ by
$\mc T(u, \varepsilon)$, i.e.,
\begin{align*}
    \|u-\mc T(u, \varepsilon)\|_{\mc V}\leq\varepsilon.
\end{align*}
The total computational work for truncating a tensor $u$ can be bounded by a
constant multiple of
$dr^4+r^2\sum_{j=1}^dn_j$,
where $r=|r(u)|_\infty$ and $n_j:=\dim(V_j)$.

We need to combine the wavelet coarsening with the tensor rank tuncation.
Recall that to apply $\textbf{COARSE}$ to a tensor $u\in\mc V$ of finite
support in the wavelet dictionary, we would have to search through all entries
of $u$, a process that scales exponentially in $d$. Thus, we require low
dimensional quantities that allow us to perform this task. For this purpose we
use \emph{contractions}\footnote{We remark that this is
a slight abuse of terminology for general tensor contractions.} introduced in
\cite{DB}. For a tensor $\bs u\in\ell_2(\J^d)$ where $\J$ is a 1D wavelet
index set, we set 
\begin{align}\label{eq:contract}
    \pi_j(\bs u)=(\pi_j(\bs u)[\lambda_j])_{\lambda_j\in\J}:=\left(\sqrt{
    \sum_{\lambda_1,\ldots,\lambda_{j-1},
    \lambda_{j+1},\ldots,\lambda_d\in\mc J^{d-1}}|\bs u_{\lambda_1,\ldots,\lambda_j,\ldots,
    \lambda_d}|^2}\right)_{\lambda_j\in\J}.
\end{align}
Recalling the restriction operator
\begin{align*}
    R_{\J_1\times\ldots\times\J_d}\bs u[\lambda]:=
    \begin{cases}
        \bs u[\lambda],&\quad\text{if }\lambda\in\J_1\times\ldots\times\J_d,\\
        0,&\quad\text{otherwise},
    \end{cases}
\end{align*}
the two important properties of these contractions are
\begin{align}\label{eq:contractqo}
    \pi_j(\bs u)[\lambda_j]&=
    \sqrt{\sum_k|\sigma^j_k|^2|\bs U_j^k(\lambda_j)|^2},\notag\\
    \|(I-R_{\J_1\times\ldots\times\J_d})\bs u\|&\leq
    \sqrt{\sum_{j=1}^d\sum_{\lambda\in\J\setminus \J_j}|\pi_j(\bs u)
    [\lambda]|^2},\notag\\
    &\leq\sqrt{d}
    \|(I-R_{\J_1\times\ldots\times\J_d})\bs u\|,
\end{align}
where $\bs U_j^k$ is the $k$-th column of the $j$-th HOSVD basis frame and
$\sigma^j_k$ are the corresponding singular values. We use the notation
\begin{align*}
    \supp_j(\bs u):=\supp(\pi_j(\bs u)),
\end{align*}
to refer to the 1D support of $\bs u$ along the $j$-th dimension, i.e.,
$\bs u$ can
be viewed as $\bs u\in\ell_2(\supp_1(\bs u)\times\cdots\times
\supp_d(\bs u))$.

\subsection{Separable Preconditioning}\label{sec:prec}

Suppose we want to solve an equation on the Sobolev space $\mc X\subset H^s(\Omega)$
on a bounded Lipschitz domain $\Omega\subset\R^d$
with appropriate boundary conditions. Typically, the point of
departure is a Riesz wavelet basis $\Psi_{L_2}$
for $L_2(\Omega)$ from which we obtain
a whole range of Riesz bases for $H^s$ by a simple diagonal scaling
(see e.g.\ \cite[Section 5.6.3]{KU}) $\Psi_{H^1}:=\bs D^{-s}\Psi_{L_2}$,
where $\bs D:=(\delta_{\lambda,\mu}\|\psi_{\lambda}\|_{H^1})_{\lambda, \mu}$. This is
equivalent to reformulating \cref{eq:opeqdisc} as the preconditioned
infinite system
\begin{align}
    \bs D^{-s}\bs A\bs D^{-s}\bs D\bs u=\bs D^{-s}\bs f.\label{eq:precop}
\end{align}

In the context of high dimensional problems, $d\gg 1$ is large and approximating
the solution to \cref{eq:opeqdisc} is in general an intractable problem
(see, e.g., \cite{Novak}). However, given
a product structure of the domain $\Omega=\times_{j=1}^d\Omega_j$ (or smooth
images thereof), the problem \cref{eq:opeqdisc} can be solved with tractable
(algebraic) methods (see, e.g., \cite{DDGS}).
For this we will need $\Psi$ to be a
tensorised basis of lower dimensional components, i.e.,
$\Psi:=\times_{j=1}^d\Psi_j$ and we reconsider $\mc X$ as a tensor space
$\mc X=\otimes_{j=1}^d\mc X_j$. This way, if $A$ permits a separable
structure or can be well approximated in such a form, than we can
discretize $A$ such that it
preserves the product structure with low dimensional components.

Unfortunately, the space $H^s(\Omega)$ is not equipped with a cross norm,
i.e., for an elementary tensor product $v=v_1\otimes\cdots\otimes v_d$
\begin{align*}
    \|v\|_s\neq \|v_1\|_s\cdots \|v_d\|_s.
\end{align*}
Considering again \cref{eq:precop}, this means that $\bs D^{-s}$ can not be
represented in a separable form. However, this issue was addressed in \cite{MarkusSobolev},
where the exact preconditioning $\bs D^{-s}$ was replaced by an approximate
separable scaling via exponential sum approximations. We will utilize this
separate scaling both for preconditioning the Galerkin solver and the approximate
residual evaluation. We briefly recall some basic properties of the said
preconditioning\footnote{For ease of presentation,
we restrict ourselves to $s=1$.}.

For certain parameters $\delta>0$, $\eta>0$, $T>1$,
we choose 
$h\in\left(0, \frac{\pi^2}{5(|\ln(\delta/2)|+4)}\right)$, 
$n^+\geq h^{-1}\max\left\{\frac{4}{\sqrt{\pi}},\sqrt{|\ln(\delta/2)|} \right\}$
and $n\geq h^{-1}\left(\ln\frac{2}{\sqrt{\pi}}+|\ln(\min\{\delta/2,\eta\})|+ \frac{1}{2}\ln T\right)$
The approximation involved is
\begin{align*}
    \frac{1}{\sqrt{t}}&=\frac{2}{\sqrt{\pi}}\int_\R
    \frac{e^{-t\ln^2(1+e^x)}}{1+e^{-x}}dx
    \approx\sum_{k=-n}^{n^+}hw(kh)e^{-\alpha(kh)t}=:\varphi_{n^+,n}(t),
\end{align*}
where $w(x):=\frac{2}{\sqrt{\pi}}(1+e^{-x})^{-1}$, $\alpha(x):=\ln^2(1+e^x)$ and 
$t>0$ is some scaling weight.
We get
\begin{align*}
    \left|\frac{1}{\sqrt{t}}-\varphi_{n^+,n}(t)
    \right|\leq\frac{\delta}{\sqrt{t}},
    &&
    |\varphi_{n^+,\infty}(t)-\varphi_{n^+,n}(t)| \leq\frac{\eta}{\sqrt{t}},
\end{align*}
for all $t\in[1, T]$. For the exact diagonal preconditioning, the scaling
weights for  tensor product wavelets can be obtained by observing
that $H^1$ (and similarly $H^s$) is isomorphic to the intersection of Hilbert spaces
\begin{align*}
    H^1(\Omega)\cong \bigcap_{j=1}^dL_2(\Omega_1)\otimes\cdots
    \otimes H^1(\Omega_j)\otimes\cdots\otimes L_2(\Omega_d),
    \qquad \text{with }
    \Omega=\Omega_1\times\cdots\times\Omega_d.
\end{align*}
The norm on the intersection
space leads to the scaling weight $t:= \sum_{j=1}^d\|\psi_{\lambda_j}\|^2_{H^1}$, 
for $\psi_\lambda=\otimes_{j=1}^d\psi_{\lambda_j}$. We will denote by
$$
\bs S(\delta,\eta)
\qquad\text{and}\qquad
\bs S(\delta) := \lim_{\eta\rightarrow0} \bs S(\delta,\eta)
$$
the corresponding separable approximation to $\bs D$ and the limit, respectively. 
We mention important properties from \cite{MarkusSobolev} for later use
\begin{subequations}
\begin{align}
    \|\bs D\bs S^{-1}(\delta,\eta)\|&\leq 1+\delta,\quad\forall\eta>0,\\
    \label{eq:ds}\|\bs D\bs S^{-1}(\delta)\|&\leq 1+\delta,\\
    \|\bs S(\delta)\bs D^{-1}\|&\leq\frac{1}{1-\delta},\\
    \|\bs D(\bs D^{-1}-\bs S^{-1}(\delta,\eta))R_{\J_T}
    \|&\leq\delta,\quad\forall\eta>0,\\
    \|\bs D(\bs S^{-1}(\delta)-\bs S^{-1}(\delta,\eta))R_{\J_T}\|
    &\leq\eta,\quad\forall\delta>0,\\
    \label{eq:ids}\|\bs S(\delta)(\bs S^{-1}(\delta)-\bs S^{-1}(\delta,\eta))R_{\J_T}\|
    &\leq\frac{\eta}{1-\delta},\\
    \bs S^{-1}(\delta,\eta)&\leq\bs S^{-1}(\delta),\quad\forall\eta>0,\\
    1-\delta&\leq \bs S^{-1}(\delta)\bs D\leq1+\delta,\\
    1-\delta&\leq \left(\bs S^{-1}(\delta,\eta)\bs D\right)_{
    \lambda\in\J_T}\leq1+\delta,\quad\forall\eta>0,
\end{align}
\end{subequations}
where the last three inequalities are to be understood componentwise
and $T$ has to be chosen large enough in dependence on
$\J_T$. We thus
seek to approximate the solution of the
separably\footnote{Though $\bs A^\delta$ is still not separable,
it can be well approximated by separable operators.}
preconditioned equation
\begin{align*}
    \bs S^{-1}(\delta)\bs A\bs S^{-1}(\delta)\bs S(\delta)
    \bs u=\bs A^\delta
    \bs u^\delta=\bs f^\delta=\bs S^{-1}(\delta)\bs f,
\end{align*}
with the shorthand notation
\begin{align*}
    \bs S^{-1}(\delta)\bs A\bs S^{-1}(\delta)=:\bs A^\delta,\quad
    \bs S^{-1}(\delta)\bs f=:\bs f^\delta,\quad
    \bs S(\delta)\bs u=:\bs u^\delta.
\end{align*}

\section{Perturbed finite-dimensional descent method}\label{sec:pdescent}

For further presentation we formulate a general descent method with
perturbations for solving the linear system
$Ax=b$, $A:\mc V\rightarrow\mc V$, $x,b\in\mc V$, 
where $A$ is an s.p.d.\ matrix and $\mc V$ is a finite-dimensional vector
space, possibly an algebraic tensor space with $N:=\dim(\mc V)<\infty$, i.e., $\mc V \cong\R^N$.

For simplicity of presentation, we omit preconditioning at this point. The
analysis for the case of exact preconditioning remains the same.
Approximate preconditioning adds a perturbation to the descent direction.

We will frequently use the associated quadratic functional
\begin{align*}
    f(x) := \frac{1}{2}\innerp{x, Ax}-\innerp{b, x}
    \equiv f_{A,b}(x), 
\end{align*}
where $\innerp{\cdot, \cdot}$ denotes the standard inner product on $\mc V$ with induced 
Euclidean norm $\|\cdot\|$ and $\|\cdot\|_A:=\innerp{\cdot, A\cdot}$ the energy norm. The very well-known descent method then reads as follows.

\begin{algorithm}
 \caption{Descent method for minimizing $f(x)$.}
\label{descalg}
\begin{algorithmic}[1]
    \REQUIRE $x^{(0)}\in\mc V$
    \STATE $k\leftarrow 0$
    \WHILE{stopping criterion for $f(x^{(k)})$ not satisfied}
        \STATE choose/update descent direction $d^{(k)}$ \label{desc:dk0}
        \STATE $d^{(k)}\leftarrow d^{(k)}+\eps^{(k)}$ (e.g., truncation) \label{desc:dk}
        \STATE compute step size $\alpha_k$ \label{desc:ak}
        \STATE $x^{(k+1)}\leftarrow x^{(k)}+\alpha_kd^{(k)}$
        \STATE $x^{(k+1)}\leftarrow x^{(k+1)}+\epss^{(k+1)}$ \label{desc:xk}
        \STATE $k\leftarrow k+1$
    \ENDWHILE
\end{algorithmic}
\end{algorithm}

In a tensor based solver, lines \ref{desc:dk} and \ref{desc:xk} are 
typical candidates for truncating a tensor due to the increase in ranks
after the summation. The quantities 
$\varepsilon^{(k)}_j$, $j=1,2$, represent the error incurred due to truncation, where $x^{(k)}$ is replaced by a truncated version
$\tilde{x}^{(k)}:=\trunc{x^{(k)}}{\varepsilon}$, such that $\|\epss^{(k)}\|\leq \varepsilon$ for the truncation error $\epss^{(k)}:=\tilde{x}^{(k)}-x^{(k)}$. We emphasize
that the analysis has to rely solely on the control of the
\emph{magnitude} of $\epss^{(k)}$ without restricting
the \emph{direction} of $\epss^{(k)}$,
which destroys optimality features of conjugate directions.

\subsection{Gradient descent}
Choosing $d^{(k)}=r^{(k)}+\epss^{(k)}$, with $r^{(k)}:=b-Ax^{(k)}$ being the residual,
and using the optimal step size $\alpha_k$ leads
to the well-known gradient-type descent method. The following proposition shows
that appropriately choosing $\eps^{(k)}$ and $\epss^{(k)}$ ensures
the same asymptotic convergence as the exact gradient descent method.

\begin{proposition}\label{Prop:rateg}
    For the choice $d^{(k)}=r^{(k)}$ in line \ref{desc:dk0} and
    \begin{align*}
        \alpha_k=\argmin_{\alpha\in\R}f(x^{(k)}+\alpha d^{(k)})
    \end{align*}
    in line \ref{desc:ak}
    (exact line search)
    of \cref{descalg}, we have the estimate for the error
    $e^{(k)}:=x^*-x^{(k)}$
    \begin{align}
        \|e^{(k)}\|_{A}&\leq\theta^{k}\|e^{(0)}\|_A+
        \sum_{j=0}^{k-1}\theta^{k-j-1}\left(\frac{\|\eps^{(j)}\|_A}{\lambda_{\min}}+
        \|\epss^{(j+1)}\|_A\right),\label{eq:gradient}
    \end{align}
    with reduction factor
     $\theta :=\frac{\lambda_{\max}-\lambda_{\min}}{\lambda_{\max}+\lambda_{\min}}$,
    and $\lambda_{\max}$ and $\lambda_{\min}$ being the largest and smallest
    eigenvalues of $A$, respectively.
\end{proposition}
\begin{proof}
    It holds that the iterate $x^{(k+1)}$ can be written as 
    $x^{(k+1)} = x^{(k)}+\alpha_k r^{(k)}+\alpha_k\eps^{(k)}+\epss^{(k+1)}$
    and the error reads 
    $e^{(k+1)}=(I-\alpha_k A)e^{(k)}+\alpha_k\eps^{(k)}+\epss^{(k+1)}$.
    The optimal step size is known to be
    $\alpha_k=\frac{\innerp{d^{(k)}, d^{(k)}}}{\innerp{d^{(k)}, Ad^{(k)}}}$. 
    Let $\{\lambda_j\}_{j=1,\ldots, N}$ denote the eigenvalues of $A$ and
    $\{\psi_j\}_{j=1,\ldots,N}$ the corresponding orthonormal basis of eigenvectors.
    Since $A$ is s.p.d., we get the standard estimate
    ($c_j:=\innerp{d^{(k)}, \psi_j}$)
    \begin{align}
        \innerp{d^{(k)}, Ad^{(k)}}=\innerp{\sum_{j=1}^N c_j\psi_j,
        \sum_{j=0}^N c_j\lambda_j\psi_j}=\sum_{j=1}^N
        \lambda_jc_j^2
        \geq
        \lambda_{\min}\|d^{(k)}\|^2.
        \label{eq:step}
    \end{align}
    Using standard arguments for the analysis of the gradient descent method
    (cf.\ \cite[Thm.\ 9.2.3]{Hackbusch93}), we get
     $\|e^{(k+1)}\|_A\leq \theta\|e^{(k)}\|_A+\frac{\|\eps^{(k)}\|_A}{\lambda_{\min}}
     +\|\epss^{(k+1)}\|_A$,
    which proves \cref{eq:gradient}.
\end{proof}


\subsection{Conjugate gradient descent}\label{sec:cg}
The (rank-)truncated (P)CG method was first proposed by C.\ Tobler in
\cite[Algorithm 9]{ToblerPhd} with promising
numerical results.

\begin{algorithm}
 \caption{Truncated (P)CG method}
\label{pcgalg}
\begin{algorithmic}[1]
    \REQUIRE $x^{(0)}\in\mc V$
    \STATE $r^{(0)}\leftarrow b-Ax^{(0)}$, $d^{(0)} \leftarrow r^{(0)}+\eps^{(0)}$
    \STATE $k\leftarrow 0$
    \WHILE{stopping criterion for $f(x^{(k)})$ not satisfied}
        \STATE $\alpha_k \leftarrow \frac{\innerp{r^{(k)}, d^{(k)}}}{\innerp{d^{(k)}, Ad^{(k)}}}$, \label{pcg:ak}
    	\STATE $x^{(k+1)} \leftarrow x^{(k)}+\alpha_k d^{(k)}+\epss^{(k+1)}$ \label{pcg:xk}
	\STATE  $\beta_k\leftarrow -\frac{\innerp{r^{(k+1)}, Ad^{(k)}}}{\innerp{d^{(k)}, Ad^{(k)}}}$
	\STATE $d^{(k+1)} \leftarrow r^{(k+1)} + \beta_k d^{(k)} + \eps^{(k+1)}$
	\STATE $k\leftarrow k+1$
    \ENDWHILE
\end{algorithmic}
\end{algorithm}

Obviously, the perturbed CG method does not preserve orthogonality of the search
directions w.r.t.\ $\innerp{\cdot,A\cdot}$
and the resulting algorithm is not a Krylov method (see also below). Nevertheless, we can guarantee the perturbed CG
to be a descent method which in turn will provide us with a convergence estimate. 

\begin{lemma}\label{lemma:angle}
    Let $\kappa:=\frac{\lambda_{\max}}{\lambda_{\min}}$ and fix some  
    $\tau\in(0,\frac{1}{\sqrt{1+\kappa^2}})$. Let
    $\delta_1, \delta_2>0$ and $\gamma>0$ be chosen such that
 	$\frac{3}{2}\delta_1+\delta_2\leq \frac{1}{\tau^2}-(1+\kappa^2)$ and $(1-\frac{\delta_1}{2})\tau\geq\gamma$.
    If the error sequence $\eps^{(k)}$ satisfies
    \begin{align}
        \|\eps^{(k)}\|\leq\min\left\{\frac{\delta_1}{2},
        \frac{\delta_2\|r^{(k)}\|}{2|\beta_{k-1}|\, \|d^{(k-1)}\|}
        \right\}\|r^{(k)}\|\label{eq:errseq},
    \end{align}
    then $d^{(k)}$ is a descent direction with 
    $\innerp{r^{(k)}, d^{(k)}}\geq\gamma\|r^{(k)}\|\|d^{(k)}\|$, 
    where the angle $\gamma$ does not depend on $k$. 
\end{lemma}
\begin{proof}
    First we show that $\|r^{(k)}\|\geq\tau\|d^{(k)}\|$. To this end, note that
    \begin{align}
        \|d^{(k)}\|^2=\innerp{d^{(k)}, d^{(k)}}
                 &=\|r^{(k)}\|^2
                 	+ \beta_{k-1}^2\|d^{(k-1)}\|^2+\|\eps^{(k)}\|^2
        			+ 2\beta_{k-1}\innerp{r^{(k)}, d^{(k-1)}}
			\nonumber \\
        		&\quad+2\innerp{r^{(k)}, \eps^{(k)}}
			+ 2\beta_{k-1}\innerp{d^{(k-1)}, \eps^{(k)}}
		\label{eq:nrmdk}.
    \end{align}
    Next, we get
    \begin{align*}
        \innerp{r^{(k)}, d^{(k-1)}}
        &=\innerp{b-A(x^{(k-1)}+\alpha_{k-1}d^{(k-1)}), d^{(k-1)}}\\
        &= \innerp{r^{(k-1)}, d^{(k-1)}}-\frac{\innerp{r^{(k-1)}, d^{(k-1)}}}
        {\innerp{d^{(k-1)}, Ad^{(k-1)}}}\innerp{Ad^{(k-1)}, d^{(k-1)}}=0.
    \end{align*}
    For the term $\beta_{k-1}^2\|d^{(k-1)}\|^2$ we get
    \begin{align*}
        \beta_{k-1}^2\|d^{(k-1)}\|^2&=
        \frac{|\innerp{r^{(k)}, Ad^{(k-1)}}|^2}
        {|\innerp{d^{(k-1)}, Ad^{(k-1)}}|^2}\|d^{(k-1)}\|^2
        \overset{\cref{eq:step}}{\leq} \frac{|\innerp{r^{(k)}, Ad^{(k-1)}}|^2}
        {\lambda_{\min}^2|\innerp{d^{(k-1)}, d^{(k-1)}}|^2}\|d^{(k-1)}\|^2\\
        &\leq\frac{\lambda_{\max}^2\|r^{(k)}\|^2\|d^{(k-1)}\|^2}
        {\lambda_{\min}^2\|d^{(k-1)}\|^4}\|d^{(k-1)}\|^2\leq\kappa^2\|r^{(k)}\|^2.
    \end{align*}
    Using \cref{eq:errseq}, we estimate the term \cref{eq:nrmdk} as 
    $\|d^{(k)}\|^2\leq (1+\kappa^2+\frac{\delta_1}{2}+\delta_1+\delta_2) \|r^{(k)}\|^2
        \leq\frac{1}{\tau^2}\|r^{(k)}\|^2$. 
    This finally gives us the desired claim
    \begin{align*}
        \innerp{r^{(k)}, d^{(k)}}&=\innerp{r^{(k)}, r^{(k)}}+\innerp{r^{(k)}, \eps^{(k)}}
        \geq\innerp{r^{(k)}, r^{(k)}}-\|r^{(k)}\|\|\eps^{(k)}\|\\
        &=\|r^{(k)}\|(\|r^{(k)}\|-\|\eps^{(k)}\|)\geq\|r^{(k)}\|^2(1-\frac{\delta_1}{2})
        \geq\gamma\|r^{(k)}\|\|d^{(k)}\|.
    \end{align*}
\end{proof}

With this preparation at hand we get the following convergence estimate.

\begin{theorem}\label{thm:cg}
    Let the assumptions of \cref{lemma:angle} hold.
    For the truncation tolerance $\varepsilon_2$ set
    \begin{align}\label{eq:trunctol}
        \|\varepsilon_2^{k+1}\|\leq\theta\mu(\lambda_{\max})^{-1}\|r^{(k)}\|,
    \end{align}
    with
    \begin{align}
        \theta:=\sqrt{1-\frac{\gamma^2}{2\kappa}},
        \quad
        \gamma<1, \quad\kappa>1,\quad
        \mu<\theta^{-1}-1\label{eq:ratecg}.
    \end{align}
    Then, we have
    \begin{align}
        \|e^{(k)}\|_A\leq [\theta(1+\mu)]^k\|e^{(0)}\|_A,\label{eq:errcg}
    \end{align}
    with the error reduction factor
    \begin{align*}
        \varrho:=\theta(1+\mu)<1.
    \end{align*}
\end{theorem}
\begin{proof}
    Without loss of generality we can assume the solution is at the origin
    $x^*=0$ and thus $b=0$. Since  $d^{(k)}$
    is a descent direction by \cref{lemma:angle}, \cite[Lemma 6.2.2]{JR} yields
    $f(x^{(k+1)})\leq f(x^{(k)})-\frac{\gamma^2}{4\lambda_{\max}}\|r^{(k)}\|^2$. 
    Using an eigenbasis of $A$ as in \cref{eq:step}, we get
    \begin{align*}
        f(x^{(k)})=\frac{1}{2}\innerp{x^{(k)}, Ax^{(k)}}=\frac{1}{2}
        \sum_{j=1}^N\lambda_jc_j^2,
    &&
          \|r^{(k)}\|^2=\innerp{Ax^{(k)}, Ax^{(k)}}=\sum_{j=1}^N\lambda_j^2c_j^2.
    \end{align*}
    This gives
    \begin{align*}
        f(x^{(k+1)})\leq\frac{1}{2}\sum_{j=1}^N\lambda_jc_j^2(1-
        \frac{\gamma^2}{2\lambda_{\max}}\lambda_j)
        \leq (1-\frac{\gamma^2\lambda_{\min}}{2\lambda_{\max}})
        \frac{1}{2}\sum_{j=1}^N\lambda_jc_j^2=
        (1-\frac{\gamma^2}{2\kappa})f(x^{(k)}).
    \end{align*}
    The identity $2f(x)=\|x\|_A^2$ gives the desired claim for $\theta$ as in
    \cref{eq:ratecg}. Finally, we get with \cref{eq:trunctol}
    \begin{align*}
        \|e^{(k+1)}\|_A&\leq\theta\|e^{(k)}\|_A+\|\varepsilon_2^{(k+1)}\|_A,\\
        &\leq\theta\|e^{(k)}\|_A+\theta\mu\|e^{(k)}\|_A,\\
        &=\varrho\|e^{(k)}\|_A.
    \end{align*}
    This completes the proof.
\end{proof}

\begin{remark}\label{remark:tols}
    Note that the rate in \cref{eq:ratecg} is \emph{asymptotically} the same as in
    \cref{Prop:rateg} for large $\kappa$. This is
    not surprising, since we used the same approach for analyzing the convergence
    as in the gradient descent method. 

    Of course, \cref{eq:ratecg} is \emph{qualitatively} worse ,
    since it applies to a broader setting than
    the gradient descent method. 

    The preceding analysis is a worst case scenario that guarantees
    convergence of the method with a monotonic decrease of the error in
    the energy norm.
    However, numerically, the perturbed CG performs
    far better than the gradient descent method. This is due to the fact
    that the perturbed CG inherits some nice properties of its exact counterpart,
    as can be seen in the following lemma. Moreover, the analysis in
    \cref{thm:cg} is quite general,
    since we only require \emph{local optimality} (i.e., a descent direction)
    and the resulting bound in \cref{eq:errcg} is thus by no means optimal.

    Note, that according to \cref{eq:errcg}, the truncation tolerance
    $\|\epss^{(k+1)}\|$ should be set proportional to $\theta\|e^{(k)}\|_A$. However,
    since the error reduction factor $\theta$ corresponds to a worst case scenario, 
    this tolerance might be unnecessarily prohibitive 
    and significantly hamper quantitative
    performance. 

    A more detailed look on the estimates from
\cite[Lemma 6.2.2]{JR} reveals 
	$f(x^{(k+1)})\leq \frac{1}{2}\sum_{j=1}^N\lambda_jc_j^2-\alpha_k
        \innerp{r^{(k)}, d^{(k)}}$, 
    which suggests to choose an adaptive tolerance proportional to $\alpha_k\|d^{(k)}\|$. This is precisely the case for the adaptive tolerance
    strategy by C.\ Tobler in \cite[Algorithm 9]{ToblerPhd}.
    Hence, we use this in our subsequent numerical experiments.
\end{remark}

\begin{lemma}\label{lemma:polycg}
    For the perturbed CG method we have the following representations
    \begin{align*}
        r^{(k)}&=(I-A\,p^{(k)}(A))r^{(0)}
        	-A\left(\sum_{j=0}^{k-1}q_{k-j-1}^{(k)}(A)\eps^{(j)}
        +A\sum_{j=1}^{k}g^{(k)}_{k-j}(A)\epss^{(j)}\right),\\
        e^{(k)}&=(I-A\,p^{(k)}(A))e^{(0)}
        -\left(\sum_{j=0}^{k-1}q_{k-j-1}^{(k)}(A)\eps^{(j)}
        +\sum_{j=1}^{k}g_{k-j}^{(k)}(A)\epss^{(j)}\right),
    \end{align*}
    where $p^{(k)}\in\cP_{k-1}$, i.e., a polynomial of degree $k-1$,
    $g^{(k)}_j\in\cP_j$ with $g^{(k)}_j(0)=1$, $j=0,\ldots, k-1$, 
    and $q_j^{(k)}\in\cP_j$ such that
    $p^{(k)}(t)=\sum_{j=0}^{k-1}q_j^{(k)}(t)$.
\end{lemma}
\begin{proof} We prove the assertion by induction over $k$. For $k=1$, we have $x^{(1)}=x^{(0)}+\alpha_0(r^{(0)}+\eps^{(0)})+\epss^{(1)}=x^{(0)}+\alpha_0Ar^{(0)}+\alpha_0\eps^{(0)}+\epss^{(1)}$. As a consequence, $r^{(1)}=b-Ax^{(1)}=(I-\alpha_0A)r^{(0)}-\alpha_0A\eps^{(0)}-A\epss^{(1)}$ and 
        \begin{align*}
            d^{(1)}&=r^{(1)}+\beta_0d^{(0)}+\eps^{(1)}
            	=(I-\alpha_0A)r^{(0)}-\alpha_0A\eps^{(0)}-A\epss^{(1)}
            		+\beta_0(r^{(0)}+\eps^{(0)})+\eps^{(1)}\\
            &=(I+\beta_0I-\alpha_0A)r^{(0)}+(\beta_0I-\alpha_0A)\eps^{(0)}+\eps^{(1)}
            		-A\epss^{(1)},
        \end{align*}
 from which the assertion follows for $k=1$. Now, let the claim hold for some $k\ge 1$, then, we get by induction that
        \begin{align*}
            x^{(k+1)}&=x^{(k)}+\alpha_kd^{(k)}+\epss^{(k+1)},\\
                   	&=x^{(0)}+p^{(k)}(A)r^{(0)}
				+ \sum_{j=0}^{k-1} q_{k-j-1}^{(k)}(A)\eps^{(j)}
				+\sum_{j=1}^{k} g_{k-j}^{(k)}(A)\epss^{(j)}\\
            		&\qquad +\alpha_k\left(\pp^{(k)}(A)r^{(0)}
				+\sum_{j=0}^{k} \qq_{k-j}^{(k)}(A)\eps^{(j)}
				+A\sum_{j=1}^{k} \gt_{k-j}^{(k)}(A)\epss^{(j)}\right)
            			+\epss^{(k+1)}\\
            		&=x^{(0)}+p^{(k+1)}(A)r^{(0)}
				+\sum_{j=0}^{k}q_{k-j}^{(k+1)}(A)\eps^{(j)}
            			+\sum_{j=1}^{k+1}g_{k-j+1}^{(k+1)}(A)\epss^{(j)},
        \end{align*}
        with $p^{(k+1)}:= p^{(k)}+\alpha_k\pp^{(k)}$, $q_j^{(k+1)}:=q_j^{(k)}+\alpha_k\qq_j^{(k)}$ for $j<k$ and $q_k^{(k+1)}:=\alpha_k\qq_k^{(k)}$ as well as $g_j^{(k+1)}:= g_j^{(k)}+\alpha_kt\gt^{(k)}_{j-1}$ for $j>0$ and $g_{0}^{(k+1)}:=1$. Note that the properties stated
        in this Lemma hold for the polynomials
        $p^{(k+1)}$, $q_j^{(k+1)}$ and $g_j^{(k+1)}$. Finally,
	\begin{align*}
            d^{(k+1)}
            &= r^{(k+1)}+\beta_k d^{(k)}+\eps^{(k+1)} \\
            &= (I-A\,p^{(k)}(A))r^{(0)}
            	-A\sum_{j=0}^{k}q_{k-j}^{(k)}(A)\eps^{(j)}
            	-A\sum_{j=1}^{k+1}g_{k-j+1}^{(k)}(A)\epss^{(j)}\\
            &\quad
            	+\beta_k\left(\pp^{(k)}(A)r^{(0)}
		+\sum_{j=0}^{k} \qq_{k-j}^{(k)}(A)\eps^{(j)}
		+A\sum_{j=1}^{k}\gt_{k-j}^{(k)}(A)\epss^{(j)}\right)
            	+\eps^{(k+1)}\\
            &=
            	(I-Ap^{(k)}(A)
		+\beta_k\pp^{(k)}(A))r^{(0)}
		+\eps^{(k)}
		-A \sum_{j=0}^{k}q_{k-j}^{(k)}(A)\eps^{(j)}\\
	&\quad
            	+\beta_k\sum_{j=0}^{k}\qq_{k-j}^{(k)}(A)\eps^{(j)}
            	-A\epss^{(k+1)}-A\sum_{j=1}^k(g_{k-j+1}^{(k)}(A)
		-\beta_k \gt_{k-j}^{(k)}(A))\epss^{(j)},
        \end{align*}
        which completes the proof.
\end{proof}

Similar to its exact counterpart, the perturbed (P)CG is thus a polynomial method
both in the initial residual and in the perturbations. It is easy to see
that the polynomials $\{p^{(j)}\}_j$ are \emph{not}
orthogonal w.r.t.\ the discrete inner product
$\innerp{p, q}_{\CG}:=\innerp{p(A)r^{(0)}, q(A)r^{(0)}}$, see \cite[Example 2.4.8]{Fischer}. 
Consequently the resulting iterates do not minimize 
$\innerp{p^{(k)}, t^{-1}p^{(k)}}_{\CG}=\|e^{(k)}\|_A^2$. 

Though the perturbed CG is a straightforward extension of its exact counterpart,
its not a constructive\footnote{By `constructive' we refer to methods which
are derived from optimization problems, such as the exact CG method is derived
by minimizing the energy norm of the error.} method,
in particular, it is not a Krylov method,
and thus standard notions of optimality are lost.

One could try to improve the estimates by considering a different inner
product in order to obtain orthogonal polynomials. However, since one has
no control over the directions of the perturbations, this route does not seem
to be promising.

\section{HTucker-Adaptive Wavelet-Galerkin Method (HT-AWGM)}\label{sec:awgm}

As already said earlier, the new HT-AWGM relies on the strategy
\begin{align*}
    \cdots\rightarrow\textbf{SOLVE}
    \rightarrow\textbf{ESTIMATE}
    \rightarrow\textbf{MARK and REFINE}
    \rightarrow\cdots
\end{align*}
which is analogous to an adaptive FEM solver. We detail the ingredients as follows.

\subsection{\textbf{SOLVE}}
We use a Galerkin solver based on the CG iterations
described in \S\ref{sec:cg} with the
approximate separable preconditioning from \cite{MarkusSobolev}, see \S\ref{sec:prec}. 
The arising procedure is referred to as
$$\textbf{PCG}(\bs S^{-1}(\delta), \bs A^\delta, \bs f^\delta,
\bs u^{(0)}, \Lambda, \varepsilon),$$
where
$\bs S^{-1}(\delta)$ is the preconditioning operator,
$\bs A^\delta$ is the discrete (infinite dimensional) operator,
$\bs f^\delta$ is the right hand side,
$\bs u^{(0)}$ is the initial guess,
$\Lambda$ is a finite index set on which the iterations are performed
and $\varepsilon$ is the residual tolerance.

\begin{remark}
We shall assume a separable structure for the operator $A$ and thus
will not discuss the approximation of more general operators
(for this see, e.g., \cite{DB}).
Thus, evaluating $\bs A_{\Lambda}\bs u$ on a finite set $\Lambda$ boils down to applying the
low dimensional components of $\bs A_{\Lambda}$ to the leafs of $\bs u$.
For the low dimensional evaluation we use the
evaluation procedures from \cite[Chapter 6]{KestlerPhd}.
\end{remark}

\subsection{\textbf{ESTIMATE}}
For this step we need a procedure for approximate residual
evaluation. This requires determining an extended index set
$\tilde{\Lambda}\supset\Lambda$ based on a desired tolerance $\varepsilon>0$
and evaluating
$$\|R_{\tilde{\Lambda}}\left(\bs f^\delta-
\bs A^\delta E_{\Lambda}\bs u_{\Lambda}\right)\|.$$
Again, due to the separable structure of $\bs A$, we only need to build
$\tilde{\Lambda}=\tilde{\Lambda}_1\times\cdots\times\tilde{\Lambda}_d$
from the low dimensional components $\tilde{\Lambda}_j$, $j=1,\ldots,d$.
For this purpose, we use the
method from \cite[Chapter 7]{KestlerPhd}.
Additionally, we need to approximate scaling
$\bs S^{-1}(\delta)$, which we discuss in detail later.
We refer to this procedure as
$$\textbf{RES}(\bs S^{-1}(\delta), \bs A^\delta, \bs f^\delta,
\bs u^\delta, \varepsilon),$$ where $\varepsilon$ refers to the relative
accuracy in the sense that
\begin{align*}
    \|\left(\bs f^\delta-
    \bs A^\delta E_{\Lambda}\bs u_{\Lambda}\right)-\tilde{\bs r}\|\leq
    \varepsilon\|\tilde{\bs r}\|,
\end{align*}
and $\tilde{\bs r}$ is the approximate residual.

\subsection{\textbf{MARK and REFINE}}
In AFEM, one first marks certain elements, which are then refined by a chosen
strategy: \textbf{REFINE}. In AWGM, these steps are performed together.
The current index set $\Lambda$ is extended, which drives
the adaptivity of the algorithm. We  use a standard bulk chasing
strategy with a parameter $\alpha\in (0,1)$, described as follows.
Suppose the current approximation
$\bs u$ is supported on $\Lambda$, then we determine a (minimal) set
$\tilde{\Lambda}\supset\Lambda$ on which the approximate residual evaluation is performed.
Then, we  compute an intermediate set $\bar{\Lambda}$ with
$\Lambda\subset\bar{\Lambda}\subset\tilde{\Lambda}$ such that
\begin{align}\label{eq:bulk}
    \|R_{\bar{\Lambda}}\bs r\|\geq\alpha\|\bs r\|,
\end{align}
where $\bs r$ is the approximate residual supported on $\tilde{\Lambda}$.

In a low dimensional setting, \cref{eq:bulk} is realized 
by an approximate sorting of the entries in $\bs r$ and forming $\bar{\Lambda}$ by the minimal number
of largest entries that satisfy \cref{eq:bulk}. Such an approach is clearly not 
feasible for large dimensions $d\gg 1$. 

In the tensor setting we can only use low dimensional quantities and thus determine
$\bar{\Lambda}$ by sorting the contractions $\pi_j(\bs r)$ using
the \textbf{COARSE} routine from the low dimensional setting, where
\textbf{COARSE}($\bs u$, $\varepsilon$) returns a tensor
$\bs v$ with $\|\bs u-\bs v\|\leq\varepsilon.$
We refer to the resulting procedure as
$$\textbf{EXPAND}(\Lambda, \bs r, \alpha).$$

\subsection{HT-AWGM Algorithm}
We now have all algorithmic ingredients at hand to describe a general AWGM procedure based on a tensor
format in \Cref{htawgm}. We use the notation $\mc C(\bs u,\varepsilon)$ to
denote $\textbf{COARSE}(\bs u, \varepsilon)$; $\mc T(\bs u, \varepsilon)$ to
denote truncation and
\begin{align*}
    \|(\bs A^\delta)^{-1}\|\leq \lambda_{\min},\quad
    \|\bs A^\delta\|\leq\lambda_{\max}
\end{align*}
\begin{algorithm}
    \caption{\textbf{HT-AWGM}}
    \label{htawgm}
\begin{algorithmic}[1]
    \REQUIRE Tolerance $\varepsilon>0$, initial finite
             index set $\Lambda^{(0, 0)}\ne\emptyset$, $\delta>0$,
             $\alpha\in (0,1)$,
             $\omega_0, \omega_1, \omega_2, \omega_3, \omega_4, \omega_5>0$,
             $M\in \N$.
    \STATE   $\bs u^{(0, 0)}\leftarrow 0$,
             $\bs r^{(0, 0)}\leftarrow\omega_0$,
             $\omega_0^{(0)}\leftarrow\omega_0$
    \FOR{$k=0,\ldots$}
        \FOR{$m=0,\ldots,M$}
            \STATE  $\bs u^{(k, m+1)}\leftarrow
                    \textbf{PCG}(\bs S^{-1}(\delta), \bs A^\delta, \bs f^\delta,
                    \bs u^{(k, m)}, \Lambda^{(k, m)}, \omega_2\|\bs r^{(k, m)}\|)$
                    \label{htawgm:pcg}
            \STATE  $\bs r^{(k, m+1)}\leftarrow
                    \textbf{RES}(\bs S^{-1}(\delta), \bs A^\delta, \bs f^\delta,
                    \bs u^{(k, m+1)},\omega_1)$
                    \label{htawgm:res}
            \IF{$(1+\omega_1)\|\bs r^{(k, m+1)}\|\leq\varepsilon$}
                \RETURN $\bs u_\varepsilon\leftarrow\bs u^{(k, m+1)}$
            \ENDIF
            \IF{$(1+\omega_1)\|\bs r^{(k, m+1)}\|\leq
                \omega_3\omega_0^{(k)}$, or
                $m=M$}\label{iftol}
                \STATE  $\bs u^{(k+1, 0)}\leftarrow \mc T(\bs u^{(k, m+1)},
                        \omega_4\lambda^{-1}_{\min}\omega_0^{(k)})$
                        \label{htawgm:trunc}
                \STATE  $\bs u^{(k+1, 0)}\leftarrow
                        \mc C(\bs u^{(k+1, 0)}, \omega_5\lambda^{-1}_{\min}
                        \omega_0^{(k)})$
                        \label{htawgm:coarse}
                \STATE  $\Lambda^{(k+1, 0)}\leftarrow\supp(\bs u^{(k+1, 0)})$
                \STATE  $\bs r^{(k+1, 0)}\leftarrow
                        \textbf{RES}(\bs S^{-1}(\delta), \bs A^\delta, \bs f^\delta,
                        \bs u^{(k+1, 0)},\omega_1)$
                \STATE  $\omega_0^{(k+1)}\leftarrow (\omega_3
                        +\omega_4+\omega_5)\omega_0^{(k)}$
                \BREAK
            \ENDIF
            \STATE  $\Lambda^{(k, m+1)}\leftarrow
                    \textbf{EXPAND}(\Lambda^{(k, m)},
                    \bs r^{(k, m+1)}, \alpha)$
        \ENDFOR
    \ENDFOR
\end{algorithmic}
\end{algorithm}
The involved parameters have the following meaning:
\begin{itemize}
    \item $\omega_0$ is the initial estimate for the right hand side, i.e.,
                     $\omega_0\geq\|\bs f^\delta\|$,
    \item $\omega_1$ is the relative precision of the residual evaluation,
    \item $\omega_2$ drives the tolerance for the approximate Galerkin solutions,
    \item $\omega_3$ is the required error reduction rate before
                     truncation and coarsening,
    \item $\omega_4$ drives the truncation tolerance that controls rank growth,
    \item $\omega_5$ drives the coarsening tolerance that controls index
                     set growth and influences rank growth by controlling
                     the maximum wavelet level,
    \item $\alpha$   is the bulk criterion parameter that drives adaptivity.
\end{itemize}


\subsection{Convergence of HT-AWGM}
We start proving the convergence of the algorithm by
investigating the approximate residual evaluation. Two types of approximation
are involved for the operator:
a) the finite index set approximation of $\bs A$ and b) the
approximation of the exact diagonal scaling $\bs D^{-1}$,
resp.\ $\bs S^{-1}(\delta)$.

\begin{lemma}\label{lemma:res}
    Let $\bs v$ be finitely supported and let $\bs A_{\varepsilon}$ denote an
    approximation to $\bs A$ in the sense that
    $\left\|\bs D^{-1}(\bs A-\bs A_\varepsilon)\bs D^{-1}\bs v\right\| \leq\varepsilon$.
    Moreover, let $\bs f_\varepsilon$ be an approximation to $\bs f$ such that
    $\left\|\bs D^{-1}(\bs f-\bs f_\varepsilon)\right\|\leq\varepsilon$.
    Finally, assume that
    $\|\bs S^{-1}(\delta)\bs f_\varepsilon\|\leq C_{\bs f}\|\bs f^\delta\|$
    for all $\varepsilon>0$
    with $C_{\bs f}\geq 1$. Then,
    \begin{align}\label{eq:res}
        &\left\|(\bs f^\delta-\bs A^\delta\bs v)-\bs S^{-1}(\delta, \eta)
        (\bs f_\varepsilon-\bs A_\varepsilon \bs S^{-1}(\delta,\eta)\bs v)\right\|
        \notag \\
        &\leq\varepsilon(1+\delta)(2+\delta)+\frac{\eta}{1-\delta}
        \left(C_{\bs f}\|\bs f^\delta\|+2\|\bs A^\delta\|\|\bs v\|\right)
        +2\frac{(1+\delta)^2}{1-\delta}\eta\varepsilon,
    \end{align}
    with $\eta, \delta>0$.
\end{lemma}
\begin{proof}
    We begin by splitting the left-hand side of \cref{eq:res} into two parts
    \begin{align*}
        &\left\|(\bs f^\delta-\bs A^\delta\bs v)-
        \bs S^{-1}(\delta, \eta)(\bs f_\varepsilon-
        \bs A_\varepsilon\bs S^{-1}(\delta,\eta)\bs v)\right\|
        \\
        &\quad\leq \underbrace{\left\|\bs f^\delta-
        \bs S^{-1}(\delta,\eta)\bs f_\varepsilon\right\|}_{=:\text{ (I)}}
        + \underbrace{\left\|\bs A^\delta\bs v-
        \bs S^{-1}(\delta,\eta)\bs A_\varepsilon\bs S^{-1}(\delta,\eta)
        \bs v\right\|}_{=:\text{ (II)}}.
    \end{align*}
    We further split (I) as
    \begin{align*}
        \left\|\bs f^\delta-\bs S^{-1}(\delta,\eta)\bs f_\varepsilon\right\|
        \leq \|\bs S^{-1}(\delta)(\bs f-\bs f_\varepsilon)\|+
        \|(\bs S^{-1}(\delta)-\bs S^{-1}(\delta,\eta))\bs f_\varepsilon\|
    \end{align*}
    and get the first part
    $\|\bs S^{-1}(\delta)(\bs f-\bs f_\varepsilon)\|=
        \|\bs S^{-1}(\delta)\bs D\bs D^{-1}(\bs f-\bs f_\varepsilon)\|\leq
        (1+\delta)\varepsilon$,
    where the last inequality follows from the property
    $\|\bs S^{-1}(\delta)\bs D\|\leq 1+\delta.$
    For the second part in (I) we get
    \begin{align*}
        \|(\bs S^{-1}(\delta)-\bs S^{-1}(\delta,\eta))\bs f_\varepsilon\|&=
        \|(\bs S^{-1}(\delta)-\bs S^{-1}(\delta,\eta))
        \bs S(\delta)\bs S^{-1}(\delta)\bs f_\varepsilon\|\\
        &\leq
        \frac{\eta}{1-\delta}\|\bs S^{-1}(\delta)\bs f_\varepsilon\|
        \leq
        C_{\bs f}\frac{\eta}{1-\delta}\|\bs f^\delta\|,
    \end{align*}
    where we used the fact
    $\|(\bs S^{-1}(\delta)-\bs S^{-1}(\delta,\eta))\bs S(\delta)\|\leq\frac{\eta}{1-\delta}$.
    In a similar fashion, we split (II) into 2 parts
    \begin{align*}
        \text{(II)}
        &\leq
        \underbrace{\|\bs S^{-1}(\delta)(\bs A-\bs A_\varepsilon) \bs S^{-1}(\delta)\bs v\|}_{=:\text{(II.1)}}\\
        &\quad+\underbrace{\|\bs S^{-1}(\delta)\bs A_\varepsilon\bs S^{-1}(\delta)
            \bs v -\bs S^{-1}(\delta, \eta)\bs A_\varepsilon
            \bs S^{-1}(\delta, \eta)\bs v\|}_{=:\text{(II.2)}},
    \end{align*}
    and follow the proof of \cite[Proposition 15]{MarkusSobolev}. For the
    first term we get
    \begin{align*}
        \text{(II.1)}
        &=
        \|[\bs S^{-1}(\delta)\bs D]\bs D^{-1}(\bs A-\bs A_\varepsilon)
        \bs D^{-1}[\bs D\bs S^{-1}(\delta)]\bs v\|
        \leq(1+\delta)^2\varepsilon,
    \end{align*}
    where we used \eqref{eq:ds}.
    The second term $\text{(II.2)}$ involves the approximation errors
    $\|\bs S(\delta)(\bs S^{-1}(\delta)-\bs S^{-1}(\delta,\eta))\bs v\|$ and
    $\|\bs S^{-1}(\delta)(\bs A-\bs A_\varepsilon)\bs S^{-1}(\delta)\bs v\|$.
    For the former we use \eqref{eq:ids} and the latter can be bounded by
    $(1+\delta)^2\varepsilon$ as in (II.1). Altogether we get
    \begin{align*}
        \text{(II.2)}&\leq
        \frac{2\eta}{1-\delta}(\|\bs A^\delta\|\|\bs v\|+
        (1+\delta)^2\varepsilon),
    \end{align*}
    which completes the proof.
\end{proof}

For a given tolerance $\textit{tol}>0$ and a finite tensor $\bs v$, we can specify
$\varepsilon$ and $\eta$ as
\begin{align*}
    \varepsilon\leq\frac{\textit{tol}}{3(1+\delta)(2+\delta)},\quad
    \eta\leq\min\left\{\frac{1-\delta}{2},
    \frac{\textit{tol}(1-\delta)}{3(C_{\bs f}\|\bs f^\delta\|+
    2\|\bs A^\delta\|\|\bs v\|)}\right\}.
\end{align*}
By \cref{eq:res} this would ensure
\begin{align*}
    \|(\bs f^\delta-\bs A^\delta\bs v)-
    \bs S^{-1}(\delta,\eta)(\bs f_\varepsilon-\bs A_\varepsilon
    S^{-1}(\delta,\eta)\bs v)\|\leq\textit{tol}.
\end{align*}
As a consequence, given the parameter $\omega_1\in (0,1)$ from \cref{htawgm}
and some fixed
$\delta>0$, we can now use \cref{eq:res} to ensure
\begin{align}\label{eq:res_acc}
    \|(\bs f^\delta-\bs A^\delta\bs v)-
    \bs S^{-1}(\delta,\eta)(\bs f_\varepsilon-\bs A_\varepsilon
    S^{-1}(\delta,\eta)\bs v)\|
    \leq \omega_1\|\bs S^{-1}(\delta,\eta)(\bs f_\varepsilon-\bs A_\varepsilon
    S^{-1}(\delta,\eta)\bs v)\|.
\end{align}

With all the above ingredients at hand, it is now easy to prove
that \cref{htawgm} converges for an appropriate choice of parameters.
There are two main components. First, we choose $\omega_1$, $\omega_2$ and
$\alpha$ appropriately such that we ensure in each inner iteration
$m\rightarrow m+1$ of Algorithm \ref{htawgm} a guaranteed error
reduction. Second, we choose $\omega_3$, $\omega_4$ and $\omega_5$ such that
after truncation and coarsening we still ensure an error reduction for the
outer iteration $k\rightarrow k+1$.

We use the notation
\begin{align*}
    \|\cdot\|_A:=
    \innerp{\cdot, \bs A^\delta\cdot},
\end{align*}
to denote the energy norm.

\begin{proposition}\label{thm:conv}
    Let $\Lambda^{(0)}=\Lambda_1^{(0)}\times\cdots\times\Lambda_d^{(0)}$
    and all $\Lambda^{(0)}_j$ are
    assumed to have a tree structure as required in \cite[\S 6]{KestlerPhd}.
    Let the parameters satisfy $0<\omega_1<\alpha$ and
    \begin{align*}
        \omega_2<\frac{(1-\omega_1)(\alpha+\omega_1)}{1+\omega_1}
        \kappa(\bs A^\delta)^{-1}.
    \end{align*}
    This guarantees an error reduction in the inner iterations
    \begin{align}\label{eq:energyred}
        \|\bs u-\bs u^{(k, m+1)}\|_A\leq \vartheta\|\bs u-\bs u^{(k, m)}\|_A,
    \end{align}
    with
    \begin{align}\label{eq:theta}
        \vartheta:=\left(1-\left(\frac{\alpha-\omega_1}{1+\omega_1}\right)^2
        \kappa^{-1}(\bs A^\delta)+
        \left(\frac{\omega_2}{1-\omega_1}\right)^2\kappa(\bs A^\delta)
        \right)^{1/2}<1
    \end{align}
    Moreover, if $M\in \N$ is chosen such that
    \begin{align}\label{eq:M}
        M\geq M^*=M^*(\delta):=\left\lceil
        \left|\frac{\ln(\omega_3[\kappa(\bs A^\delta)]^{-1/2})}
        {\ln(\vartheta)}\right|
        \right\rceil,
    \end{align}
    and
    \begin{align}\label{eq:om35}
        \omega_3+\omega_4+\omega_5<1,
    \end{align}
    then the error decreases in each outer iteration such that
    \begin{align}\label{eq:outerred}
        \|\bs u-\bs u^{(k, 0)}\|\leq \lambda_{\min}^{-1}
        \omega_0(\omega_3+\omega_4+
        \omega_5)^k.
    \end{align}
    This ensures \cref{htawgm} terminates after at most
    $K^*M^*$
    steps, where
    \begin{align}\label{eq:K}
        K^*=K^*(\varepsilon,\delta):=\left\lceil
        \left|\frac{\ln([\varepsilon\kappa(\bs A^\delta)
        \omega_3\omega_0(1+\omega_1)]^{-1}(1-\omega_1))}
        {\ln(\omega_3+\omega_4+\omega_5)}\right|
        \right\rceil,
    \end{align}
    with the output satisfying
    \begin{align*}
        \|\bs f^\delta-\bs A^\delta\bs u_\varepsilon\|\leq\varepsilon.
    \end{align*}
\end{proposition}

\begin{proof}
    The statement in \cref{eq:energyred} with $\theta$ as in
    \cref{eq:theta} is an immediate application of
    \cite[Prop.\ 4.2]{Stevenson}. The conditions on $\alpha$,
    $\omega_1$ and $\omega_2$ ensure $0<\vartheta<1$.

    In the inner iterations we thus get for any $k$
    \begin{align*}
        \|\bs u-\bs u^{(k, m)}\|&\leq\lambda_{\min}^{-1/2}
        \|\bs u-\bs u^{(k, m)}\|_A
        \leq\lambda_{\min}^{-1/2}\vartheta^m\|\bs u-\bs u^{(k, 0)}\|_A,\\
        &\leq\sqrt{\kappa(\bs A^\delta)}\vartheta^m\|\bs u-\bs u^{(k, 0)}\|
        \leq\sqrt{\kappa(\bs A^\delta)}\vartheta^m\lambda^{-1}_{\min}
        \omega^{(k)}_0.
    \end{align*}
    The requirement \cref{eq:M} ensures
    \begin{align}\label{eq:innerred}
        \|\bs u-\bs u^{(k, M)}\|\leq
        \sqrt{\kappa(\bs A^\delta)}\vartheta^M\lambda^{-1}_{\min}\omega_0^{(k)}
        \leq\omega_3\lambda_{\min}^{-1}\omega_0^{(k)}.
    \end{align}
    Alternatively, the first if-condition in line \ref{iftol} ensures
    \begin{align*}
        \|\bs u-\bs u^{(k, m+1)}\|\leq\lambda_{\min}^{-1}
        \|\bs A^\delta(\bs u-\bs u^{(k, m+1)})\|\leq
        \lambda_{\min}^{-1}(1+\omega_1)\|\bs r^{(k, m+1)}\|
        \leq \omega_3\lambda_{\min}^{-1}\omega_0^{(k)}.
    \end{align*}
    Hence, after truncation and coarsening we obtain
    \begin{align*}
        \|\bs u-\bs u^{(k+1, 0)}\|
        &\leq
        \|\bs u-\bs u^{(k, m+1)}\|+
        \|\bs u^{(k, m+1)}-\mc T(\bs u^{(k, m+1)},
        \lambda_{\min}^{-1}\omega_4\omega_0^{(k)})\|\\
        &\quad+\|\bs u^{k+1, 0}-\mc T(\bs u^{(k, m+1)},
        \lambda_{\min}^{-1}\omega_4\omega_0^{(k)})\|\\
        &\leq \lambda_{\min}^{-1}\omega_0^{(k)}
        (\omega_3+\omega_4+\omega_5)
        =\lambda_{\min}^{-1}\omega_0
        (\omega_3+\omega_4+\omega_5)^{k+1},
    \end{align*}
    which shows \cref{eq:outerred}. Combining \cref{eq:innerred} and
    \cref{eq:outerred}, we obtain
    \cref{eq:K}. Together with \cref{eq:om35} this completes the proof.
\end{proof}


\subsection{Complexity}

The complexity in rank and discretization is controlled by the intermediate
truncation and coarsening steps in line \ref{htawgm:trunc} and
\ref{htawgm:coarse} of \cref{htawgm}. This is done in analogy to the
re-coarsening step in the non-tensor case as in, e.g., \cite{CDD1};
and to the tensor recompression and coarsening as in \cite{DB}.
In \cite{GHS} it was shown that an AWGM without re-coarsening is optimal
for a moderate choice of $\alpha$. Unfortunately, the same ideas do
not carry over to the tensor case. For a detailed discussion, see
Section \ref{sec:discussion}.

In order to capture the optimal ranks and index set size
w.r.t.\ $\bs u$, we must choose a truncation tolerance in
line \ref{htawgm:trunc} and a coarsening tolerance in line
\ref{htawgm:coarse} slightly above the error $\|\bs u-\bs u^{(k, m+1)}\|$.
In addition, since in the tensor case we can only numerically realize
quasi-optimal approximations w.r.t.\ rank and discretization,
quasi-optimality constants from
\cref{eq:hosvd} and \cref{eq:contractqo} are involved.

\begin{proposition}\label{prop:complex}
    Let $\bs u^\delta\in\mc A(\gamma)$ and
    $\pi_j(\bs u^\delta)\in\mc A_s$ for all
    $1\leq j\leq d$. Assume the sequence $\gamma$ is admissible
    \begin{align*}
        \rho(\gamma):=\sup_{n\in\N}\frac{\gamma(n)}{\gamma(n-1)}<\infty.
    \end{align*}
    Finally, let the parameters $\omega_4$, $\omega_5$ satisfy
    \begin{align}\label{eq:reqcomp}
        \omega_4>(\sqrt{2d-3})\omega_3,\notag\quad
        \omega_5>\sqrt{d}(1+\sqrt{2d-3})\omega_3.
    \end{align}
    Then the following estimates hold
    \begin{align*}
        |r(\bs u^{(k, 0)})|_\infty&\leq\gamma^{-1}
        \left(C_0
        (\omega_3+\omega_4+\omega_5)^{-k}\|\bs u^\delta\|_{\mc A(\gamma)}
        \right),\quad
        \|\bs u^{(k, 0)}\|\leq C_1\|\bs u^\delta\|_{\mc A(\gamma)},\\
        \sum_{j=1}^d\#\supp_j(\bs u^{(k, 0)})&\leq C_2
        (\omega_3+\omega_4+\omega_5)^{-k/s}\left(
        \sum_{j=1}^d\|\pi_j(\bs u^\delta)\|_{\mc A_s}\right)^{1/s},\\
        \sum_{j=1}^d\|\pi_j(\bs u^{(k, 0)})\|_{\mc A_s}&\leq
        C_3\sum_{j=1}^d\|\pi_j(\bs u^\delta)\|_{\mc A_s},
    \end{align*}
    with the constants
    \begin{align*}
        C_0&:=\frac{\lambda_{\min}\sqrt{2d-3}}{\omega_0
        (\omega_4-\omega_3\sqrt{2d-3})}\rho(\gamma),\quad
        C_1:=1+\frac{(\omega_3+\omega_4)\sqrt{2d-3}}
        {\omega_4-\omega_3\sqrt{2d-3}},\\
        C_2&:=2d
        \left(\frac{\lambda_{\min}\omega_3\sqrt{2d-3}}
        {\omega_0(\omega_4-\omega_3\sqrt{2d-3})}\right)^{1/s},\\
        C_3&:=2^s(1+3^s)+2^{4s}d^{\max(1, s)}
        \frac{1+\omega_4(\sqrt{2d-3}+
        \sqrt{d}(1+\sqrt{2d-3}))}
        {\omega_4-\omega_3\sqrt{2d-3}}.
    \end{align*}
\end{proposition}
\begin{proof}
    The proof is an application of \cite[Thm.\ 7]{DB}.
\end{proof}

The complexity requirement
\cref{eq:reqcomp} together with the convergence requirement
\cref{eq:om35} imply
 $\omega_3<[1+\sqrt{2d-3}+\sqrt{d}(1+\sqrt{2d-3})]^{-1}$.

\cref{prop:complex}
ensures the outer iterates $\bs u^{(k, 0)}$ to have quasi-optimal support size
and ranks. We first demonstrate
that the quasi-optimal support size is preserved by the inner
iterates $\bs u^{(k, m)}$.

In the estimates following in this subsection we require the basic assumption
of efficient
approximability of the right hand side, i.e.,
\begin{align}\label{eq:rhs}
    \sum_{j=1}^d\#\pi_j(\bs f^\delta_\varepsilon)\leq C\varepsilon^{-1/s}
    \left(\sum_{j=1}^d\|\pi_j(\bs f^\delta)\|_{\mc A_s}\right)^{1/s},\quad
    \sum_{j=1}^d\|\pi_j(\bs f_\varepsilon^\delta)\|_{\mc A_s}\leq C
    \sum_{j=1}^d\|\pi_j(\bs f^\delta)\|_{\mc A_s},
\end{align}
for any $\varepsilon>0$ and a constant $C>0$ independent of
$\varepsilon$.

\begin{proposition}\label{prop:optsupp}
    Assume that the one dimensional components of $\bs A$ are
    $s^*$-compressible. Let the assumptions of \cref{prop:complex} hold
    for $0<s<s^*$. Moreover, let the assumptions of \cref{thm:cg} be satisfied.
    Then the intermediate index sets satisfy
    \begin{align*}
        \sum_{j=1}^d\#\Lambda_j^{(k, m)}\leq C
        \|\bs u^\delta-\bs u^{(k, m)}\|^{-1/s}
        \left(\sum_{j=1}^d\|\pi_j(\bs u^\delta)\|_{\mc A_s}\right)^{1/s},
    \end{align*}
    for a constant $C$ independent of $k$ and $m$.
\end{proposition}
\begin{proof}
On iteration $(k,m)$ we get the following.

1. Due to \cref{thm:cg}, we can ensure an upper bound on the number
of \textbf{PCG} iterations. Let $\bs r^i_{CG}$ denote the inner \textbf{PCG}
residual at \textbf{PCG} iteration $i$ and $\bs e^i_{CG}$ the
corresponding error. Then
\begin{align*}
    \|\bs r^i_{CG}\|\leq\lambda_{\max}^{1/2}\|\bs e^i_{CG}\|_A
                    \leq\lambda_{\max}^{1/2}\varrho^k\|\bs e^0_{CG}\|_A
                    \leq\kappa^{1/2}\varrho^k\|\bs r^{(k, m)}\|_A
                    \leq\omega_2\|\bs r^{(k, m)}\|_A.
\end{align*}
Thus, the number of \textbf{PCG} iterations is bounded by
\begin{align}\label{eq:pcgit}
    i\leq I^*:=\left\lceil
    \left|\frac{\ln(\omega_2\kappa^{-1/2})}{\ln(\varrho)}\right|
    \right\rceil.
\end{align}

2. Applying the same proof as in \cite[Prop.\ 6.7]{CDD1},
together with the results from \cite[Thm.\ 8]{DB} and \cref{eq:rhs}, we obtain
\begin{align*}
    \|\pi_j(\bs u^{(k, m+1)})\|_{\mc A_s}\leq C(I^*)
    \|\pi_j(\bs u^{(k, m)})\|_{\mc A_s},
\end{align*}
for $1\leq j\leq d$.

3. Applying once more \cite[Thm.\ 8]{DB}, \cref{eq:rhs} and the above
\begin{align*}
    \|\pi_j(\bs r^{(k, m+1)})\|_{\mc A_s}&\leq C(
    \|\pi_j(\bs f^\delta)\|_{\mc A_s}+
    \|\pi_j(\bs u^{(k, m+1)})\|_{\mc A_s}),\\
    &\leq\tilde{C}(
    \|\pi_j(\bs f^\delta)\|_{\mc A_s}+
    \|\pi_j(\bs u^{(k, m)})\|_{\mc A_s}),\\
    &\leq \bar{C}(
    \|\pi_j(\bs f^\delta)\|_{\mc A_s}+
    \|\pi_j(\bs u^{(k, 0)})\|_{\mc A_s}),\\
\end{align*}
for $1\leq j\leq d$ and a constant $\bar{C}>0$ independent of $k$ or $m$.

4. Let $\mc C(\bs v, N)$ denote the routine \textbf{COARSE} retaining
$N$ terms, i.e.,\\
$\sum_{j=1}^d\#\supp_j(\pi_j(\bs v))\leq N$. Let $\mc C^o(\bs v, N)$
denote the best $N$-term approximation over product sets, such that
$\sum_{j=1}^d\#\supp_j(\pi_j(\bs v))\leq N$.
For a given $\varepsilon>0$, take $N$ to be minimal such that
\begin{align*}
    \|\bs v-\mc C^o(\bs v, N)\|\leq\varepsilon.
\end{align*}
Then by property \cref{eq:contractqo}
\begin{align*}
    \|\bs v-\mc C(\bs v, N)\|\leq\sqrt{d}
    \|\bs v-\mc C^o(\bs v,N)\|\leq\varepsilon.
\end{align*}
Consequently
\begin{align*}
    \min\left\{N:\|\bs v-\mc C(\bs v, N)\|\leq\varepsilon\right\}
    \leq
    \min\left\{N:\|\bs v-\mc C^o(\bs v, N)\|\leq
    \frac{\varepsilon}{\sqrt{d}}\right\}.
\end{align*}

5. As shown in the proof of \cite[Thm.\ 7]{DB}, the best $N$-term approximation
over product sets satisfies the property
\begin{align*}
    \min\left\{N:\|\bs v-\mc C^o(\bs v, N)\|\leq
    \varepsilon\right\}
    \leq 2d\varepsilon^{-1/s}\left(\sum_{j=1}^d
    \|\pi_j(\bs v)\|_{\mc A_s}\right)^{1/s}.
\end{align*}

Combining 3.-5.\ with \cref{prop:complex} we get the desired claim
\begin{align*}
    \sum_{j=1}^d\#\Lambda_j^{(k, m+1)}&\leq
    C(\sqrt{1-\alpha^2}\|\bs r^{(k, m+1)}\|)^{-1/s}
    \left(\sum_{j=1}^d\|\pi_j(\bs f^\delta)\|_{\mc A_s}+
    \|\pi_j(\bs u^{(k,0)})\|_{\mc A_s}\right)^{1/s}\\
    &\leq\tilde{C}
    \|\bs u^\delta-\bs u^{(k, m+1)}\|^{-1/s}
    \left(\sum_{j=1}^d\|\pi_j(\bs u^\delta)\|_{\mc A_s}\right)^{1/s},
\end{align*}
with a constant $\tilde{C}>0$ independent of $k$ or $m$.
This completes the proof.
\end{proof}

The maximum wavelet level appearing in $\Lambda^{(k, m)}$ influences
the rank of the preconditioning $\bs S^{-1}(\delta, \eta)$. To show
quasi-optimality of all arising ranks, we require the following lemma.

\begin{lemma}\label{lemma:maxlevel}
    Let the assumptions of \cref{prop:optsupp} be satisfied for
    $0<s<s^*$.
    Additionally, assume
    the data $\bs f$ and operator $\bs A$
    have excess regularity for some $t>0$
    \begin{align}\label{eq:excess}
        \|\bs D^{-1+t}_j\pi_j(\bs f_\varepsilon)\|\lesssim
        \|\bs D^{-1+t}_j\pi_j(\bs f)\|<\infty
        \|\bs D^{-1+t}_j\bs A_j\|<\infty,
    \end{align}
    for any $1\leq j\leq d$ and any $\varepsilon>0$,
    where $\bs A_j$ is the one dimensional component
    of $\bs A$. Essentially \cref{eq:excess} requires the
    one dimensional components $f$ to have
    regularity $H^{-1+t}$ and the one dimensional
    wavelet basis to have regularity
    $H^{1+t}$, which in turn ensures a slightly faster decay of the wavelet
    coefficients.

    Then, on iteration $(k, m)$ the maximum level
    arising in $\Lambda^{(k, m)}$ can be bounded by
    \begin{align*}
        t^{-1}\log_2\left(C^{kM^*I^*+m}
        \|\bs u^\delta-\bs u^{(k, m)}\|^{-1-1/2s}
        \max_j\|\bs D_j^t\bs f^\delta\|
        \left(\sum_{j=1}^d\|\pi_j(\bs u^\delta)\|_{\mc A_s}\right)^{1/2s}
        \right),
    \end{align*}
    where $C>0$ is a constant independent of $k$ and $m$, $M^*$ and
    $I^*$ are defined in \cref{eq:M} and \cref{eq:pcgit} respectively.
\end{lemma}
\begin{proof}
    We want to apply \cite[Lemma 37]{MarkusSobolev}, i.e., the maximum level
    depends on the decay of the wavelet coefficients and the size of the
    tensor.
    To this end, note that
    $\Lambda^{(k, 0)}$ is obtained by coarsening $\bs u^{(k-1, m)}$ for an $m$
    that satisfies line \ref{iftol} of \cref{htawgm}. Thus, we need
    to estimate $\|\bs D_j^t\bs u^{(k-1, m)}\|$ and
    the support size of $\bs u^{(k-1, m)}$. For the latter we apply
    \cref{prop:optsupp}.

    For the former we can apply \cite[Prop.\ 39]{MarkusSobolev}
    together with assumption \cref{eq:excess},
    since $\bs u^{(k-1, m)}$
    is a polynomial in $\bs f^\delta$ (cf.\ \cref{lemma:polycg})
    and excess regularity
    is stable under truncation or coarsening. This gives the desired claim
    for $\Lambda^{(k, 0)}$.

    The set $\Lambda^{(k, m)}$, $m>1$, is obtained by coarsening the approximate
    residual $\bs r^{(k, m)}$. Thus, as above we need to estimate
    $\|\bs D_j^t\bs r^{(k, m)}\|$ and
    the support size of $\bs r^{(k, m)}$. To this end, note that the
    approximate residual is of the form
    \begin{align*}
        \bs r^{(k, m)}=\bs S^{-1}(\delta, \eta_k)
        (\bs f_{\varepsilon_k}-\bs A_{\varepsilon_k} \bs S^{-1}(\delta,\eta_k)
        \bs u^{(k, m)}),
    \end{align*}
    for $\varepsilon_k$ and $\eta_k$ chosen according to \cref{lemma:res}.
    Applying assumption \cref{eq:excess} and \cite[Prop.\ 39]{MarkusSobolev}
    to $\bs u^{(k, m)}$, we get
    \begin{align*}
        \|\bs D_j^t\bs r^{(k, m)}\|\leq C^{kM^*I^*+m}
        \|\bs D^{t}_j\bs f^\delta\|,
    \end{align*}
    for $C>0$ independent of $k$ or $m$.

    For the support size of $\bs r^{(k, m)}$ we apply \cref{eq:rhs},
    the compressibility of $\bs A$ together with \cite[Thm.\ 8]{DB} and
    \cref{prop:optsupp}. This gives
    \begin{align*}
        \sum_{j=1}^d\pi_j(\bs r^{(k, m)})\leq C
        \|\bs u^\delta-\bs u^{(k, m)}\|^{-1/s}\left(
        \sum_{j=1}^d\|\pi_j(\bs u^\delta)\|_{\mc A_s}\right)^{1/s},
    \end{align*}
    and the desired claim follows by an application of
    \cite[Lemma 37]{MarkusSobolev}.
\end{proof}

Finally, we demonstrate quasi-optimality of all intermediate ranks.
In the following
$r(\bs A)$ and $r(\bs f)$ denote
the (finite) ranks of the non-preconditioned operator and right hand side.

\begin{proposition}\label{prop:ranks}
    Let the assumptions of \cref{prop:optsupp} and \cref{lemma:maxlevel}
    hold. Let $I^*$ from \cref{eq:pcgit}
    denote the bound on the number of \textbf{PCG} iterations.
    Then, we can bound the ranks of the
    arising intermediate iterates as
    \begin{align*}
        |r(\bs u^{(k, m)})|_\infty\leq &C|r(\bs A)|_\infty^{mI^*}
        \left[1+|\ln(\|\bs u^\delta-\bs u^{(k, m)}\|)|+
        \ln\left(\sum_{j=1}^d\|\pi_j(\bs u^\delta)\|_{\mc A_s}\right)
        \right]^{2mI^*}\times\\
        &\times\left[\gamma^{-1}\left(C
        \frac{\|\bs u^\delta\|_{\mc A(\gamma)}}
        {\|\bs u^\delta-\bs u^{(k, m)}\|}\right)+
        |r(\bs f)|_\infty\right]
        =:\hat{r},
    \end{align*}
    for a constant $C>0$ independent of $k$ or $m$.
\end{proposition}
\begin{proof}
    Applying \cref{lemma:maxlevel} and \cite[Theorem 34]{MarkusSobolev}
    we get for the rank of the preconditioner at step $(k, m)$
    \begin{align*}
        |r(\bs S^{-1}(\delta, \eta_{k, m}))|_\infty
        \leq C\left(1+|\ln(\|\bs u^\delta-\bs u^{(k, m)}\|)|+k
        +\ln\left(\sum_{j=1}^d\|\pi_j(\bs u^\delta)\|_{\mc A_s}\right)\right).
    \end{align*}
    Using \cref{thm:conv}, $k$ can be bounded by
    $1+|\ln(\|\bs u^\delta-\bs u^{(k, m)}\|)|$. Finally, since $\bs u^{(k, m)}$
    is a polynomial in $\bs f^\delta$ and $\bs u^{(k, 0)}$
    (cf.\ \cref{lemma:polycg}) and together with \cref{prop:complex}
    we get the desired claim.
\end{proof}

\begin{corollary}\label{cor:complex}
    Under the assumptions of \cref{prop:ranks} the number of
    operations to produce the iterate $\bs u^{(k, m)}$ can be bounded as
    \begin{align*}
        \mc O\left(\left[1+\left|
         \ln(\bs\varepsilon^{(k, m)})\right|\right]^{8(M^*+1)I^*}
        \left[1+\gamma^{-1}\left(C
        (\bs\varepsilon^{(k, m)})^{-1}\right)\right]^{4(M^*+1)I^*}
        (\bs\varepsilon^{(k, m)})^{-1/s}\right),
    \end{align*}
    where $\bs \varepsilon^{(k, m)}:=\|\bs u^\delta-\bs u^{(k, m)}\|$
    and $C>0$ is independent of $\bs \varepsilon^{(k, m)}$.
\end{corollary}
\begin{proof}
    The dominant part for the complexity estimate is truncation.
    For a finite tensor $\bs v$ the work for truncating is bounded by
    \begin{align*}
        d|r(\bs v)|^4_\infty+|r(\bs v)|^4_\infty\sum_{j=1}^d\#\pi_j(\bs v)
    \end{align*}
    Application of \cref{prop:ranks} yields the desired claim.
\end{proof}

\begin{remark}\label{rem:complex}
    A few remarks on \cref{cor:complex} are in order.
    \begin{enumerate}
        \item The factor
        $\varepsilon^{-1/s}$ is the work related to the approximation of the frames
        of $\bs u^\delta$. It does not dominate the
        complexity estimate.
        \item
        The factor
        $\gamma^{-1} (C\frac{\|\bs u^\delta\|_{\mc A(\gamma)}}{\varepsilon})$
        reflects the low rank approximability of $\bs u^\delta$.
        Unlike in standard AWGM
        methods, due to the heavy reliance on truncation techniques to keep ranks
        small, we can not
        expect the dependence on this factor to be linear but rather algebraic at
        best. To achieve linear complexity, if at all possible, would require
        a fundamentally different approach to approximate $u$.
        \item
        The dimension dependence on $d\gg 1$ is
        hidden in the constants and the rank growth factor
        $\gamma^{-1}(C\frac{\|\bs u^\delta\|_{\mc A(\gamma)}}{\varepsilon})$.
        In particular, approximability of $\bs f$,
        $\bs A$, $\bs u^\delta$ and the behavior
        of $\kappa(\bs A^\delta)$ determine the overall
        amount of work w.r.t.\ $d$. E.g., in \cite[Thm.\ 26]{MarkusSobolev}, the
        authors assume $\gamma$ to be exponential in the rank $r$ and independent
        of $d$; the sparsity of frames of $\bs f$ to be independent of
        $d$ and the overall support size of $\bs f$ to grow at most linearly in
        $d$; the excess regularity to be $t$,
        $\kappa(\bs A^\delta)$ and the ranks of
        $\bs A$
        to be independent of $d$; the number of operations to compute
        $\bs f_\varepsilon$ to grow at most polynomially in $d$. With these
        assumptions, the authors show the number of required operations
        to compute $\bs u_\varepsilon$ to grow at most as
        $d^{C\ln(d)}|\ln(\varepsilon)|^{C\ln(d)}$ w.r.t.\ $d$. Here, $\ln(d)$ stems
        from the fact that the quasi-optimality of truncation and coarsening
        depends on $d$.
    \end{enumerate}
\end{remark}

\subsection{Discussion}\label{sec:discussion}
For a long time the question of optimality for classical adaptive methods
remained open. In particular, it was unclear if adaptive algorithms
recovered the minimal index set (of wavelets or finite elements) required for
the current error, up to a constant. In \cite{CDD1} the authors showed for
an elliptic problem solved via an adaptive wavelet Galerkin
routine that indeed optimality can be achieved.
Crucial for optimality was a re-coarsening step, as in line
\ref{htawgm:coarse} of
\Cref{htawgm}.
In \cite{GHS} it was shown that optimality can be attained without
a re-coarsening step by a careful choice of the bulk chasing
parameter $\alpha$. In \cite{OptFEM} the
results were extended to finite elements.

It was thus of interest for us to investigate
if we can ensure index set optimality without
the re-coarsening step in line \ref{htawgm:coarse} of
\Cref{htawgm}. By ``optimality'' we refer to the optimal
\emph{product} index set.

In short, this fails for the
current form of the algorithm. We briefly elaborate on the issue.

\subsubsection*{I}
On one hand, the choice of the bulk chasing parameter $0<\alpha<1$
is a delicate balance between optimality and convergence. In \cite{GHS} it
was shown that $\alpha<\kappa(\bs A)^{-1/2}$ ensures optimality, while
any choice $\alpha>0$ ensures convergence.

On the other hand, by the nature of high dimensional problems, if we want
to avoid exponential scaling in $d$, we have to consider each $\bs\Lambda_j$
in the product $\bs\Lambda_1\times\cdots\times\bs\Lambda_d$
separately. This leads to the necessity of aggregating information, as is
done via the contractions in \cref{eq:contract}. Such aggregation means we
can estimate magnitudes at best only up to a dimension dependent constant.
Specifically, $\sqrt{d}$ in \cref{eq:contractqo}.

Thus, for a given $\alpha>0$, computing the minimal index set would be of
exponential complexity. Computing the minimal index set via contractions for
a given $\alpha$, we
can show that the resulting set is optimal for an adjusted value of
\begin{align*}
    \tilde{\alpha}:=\sqrt{\frac{\alpha^2+d-1}{d}}.
\end{align*}
For $d>1$ this value is too close to $1$ and cannot additionally satisfy
$\tilde{\alpha}<\kappa(\bs A)^{-1/2}$ for realistic values of $\kappa(\bs A)$.

From a different perspective, suppose we use contractions to determine the
index set in the first dimension only and then iterate this procedure
over all dimensions. Choosing $\alpha<\kappa(\bs A)^{-1/2}$ ensures the
optimality of the resulting index sets. However, the final relative error
is bounded by $\sqrt{d(1-\alpha^2)}$. Hence, for realistic $\kappa(\bs A)$,
we loose convergence. The range of values for optimality and convergence
do not intersect since the additional constant $\sqrt{d}$ is larger than $1$.
This mismatch lies in the heart of the issue.

\subsubsection*{II}
Nonetheless, numerically it has been observed that the cardinality of the
index sets generated using contractions is close to optimal. Thus, we take
a closer look at the ratio between the two index sets. More formally, for
a tensor $\bs v\in\ell_2(\mc J_1\times\cdots\times\mc J_d)$ and a constant
$0<\alpha<1$, define
\begin{align*}
    NE(\alpha, \bs v)&:=\min\bigg\{N\in\N: \forall j\;
    \bs\Lambda_j\subset\mc J_j,\;
    \sum_{j=1}^d\#\bs\Lambda_j\leq N,\;
    \|R_{\bs\Lambda_1\times\cdots\times\bs\Lambda_d}\bs v\|
    \geq\alpha\|\bs v\|\bigg\},\\
    NQ(\alpha, \bs v)&:=\min\bigg\{N\in\N:\forall j\;
    \bs\Lambda_j\subset\mc J_j,\;
    \sum_{j=1}^d\#\bs\Lambda_j\leq N,\;
    \mu(\bs v, N)\leq\sqrt{1-\alpha^2}\|\bs v\|\bigg\}.
\end{align*}
where
\begin{align*}
    \mu^2(\bs v, N)&=\sum_{j=1}^d\sum_{\lambda_j\in\bs\Lambda_j}
    |\pi_j(\bs v)[\lambda]|^2,\\
    \bs\Lambda_j\text{ minimal s.t.\ }\sum_{j=1}^d\#\bs\Lambda_j &\leq N.
\end{align*}
We consider the ratio $\frac{NQ(\alpha, \bs v)}{NE(\alpha, \bs v)}$.

Suppose $\bs v$ is a finitely supported tensor with
$M:=\sum_{j=1}^d\#\supp_j(\bs v)$. For the number of discarded terms one can
show
\begin{align*}
    M-NE(\alpha, \bs v)\leq\vartheta(\bs v, \alpha, d)(M-NQ(\alpha, \bs v)+1)-1.
\end{align*}
where the constant $\vartheta$ can be bounded as
\begin{align*}
    d^{1/d}\leq\vartheta(\bs v, \alpha, d)\leq d.
\end{align*}
In order to estimate the desired ratio we would have to assume
\begin{align}\label{eq:ratioass}
    \frac{M-NQ}{M}\leq\frac{1-\frac{\vartheta-1}{M}-c}{\vartheta-c},
\end{align}
for some constant $0<c<1$. In this case we would get
\begin{align*}
    \frac{NQ}{NE}\leq\frac{1}{c}.
\end{align*}
Unfortunately, we were not able to derive satisfactory rigorous assumptions,
under which \cref{eq:ratioass} holds.

\bigskip

One can also
derive the following bounds for a candidate constant $C$ independent of $\bs v$
\begin{align}\label{eq:optc}
    C_{\text{mean}}\frac{d-1+\alpha^2}{d\alpha^{2/d}}
    \leq C\leq
    C_{\text{mean}}\frac{d-1+\alpha^2}{d\alpha^{2}}
\end{align}
The constant $C_{\text{mean}}$ behaves like the ratio between arithmetic and
geometric means of $\#\mc J_j$.

We performed numerical experiments for $d=2,3,4$ for tensors of different
sizes, varying the parameter $\alpha$. We considered both random tensors and
tensors with different structures replicating the form of a residual tensor.
In all test cases the bound\footnote{For most test cases the lower
bound was satisfied.} \cref{eq:optc} was satisfied.
Particularly for random
tensors, the lower bound is sharp, while for tensors with a ``residual like''
structure the bound seems overly pessimistic.

\subsubsection*{III}
Despite evidence suggesting otherwise, the statement $NQ/NE\lesssim 1$ is not
true in general. A simple counter example is a sequence of
diagonal tensors with a fixed norm, where most of the norm is contained in the
first few entries while the size of the tensor (and the number of non-zero
entries) grows.

One could consider an improvement on $NQ$ by adjusting the definition as
\begin{align*}
    NQ(\alpha, \bs v)&:=\min\bigg\{N\in\N:\forall j\;
    \bs\Lambda_j\subset\mc J_j,\;
    \sum_{j=1}^d\#\bs\Lambda_j\leq N,\\
    &\quad\quad\quad\quad
    \|\bs v-\mc C(\bs v, N)\|\leq\sqrt{1-\alpha^2}\|\bs v\|\bigg\}.
\end{align*}
This results in an additional complexity factor of $\log_2(N)$ which, however,
does not dominate the overall complexity. Although this does reduce $NQ/NE$,
the same counter example applies in this case as well. It is not clear to us
if and how we can rigorously avoid such pathological cases.

\subsubsection*{IV}
Last but not least, we would like to remark that avoiding the
re-coarsening step in line \ref{htawgm:coarse} is meaningful only if we can
avoid the re-truncation step in line \ref{htawgm:trunc} as well. At this
point the Galerkin step can not be viewed as a projection on a fixed manifold.
We envision a version of \textbf{HT-AWGM} where we extend and fix the tensor
tree adaptively, similar to the index set. However, even in this case the ideal
Galerkin step is a projection onto a non-linear manifold. Showing optimality
here without re-truncation would require a different approach than in the case
of index set optimality. We defer the analysis and implementation of such an
algorithm to future work.

\section{Numerical Experiments}\label{sec:numexp}

In this section, we test our implementation of \textbf{HT-AWGM}
analyzed in the previous section.
In particular, we are interested in the behavior of ranks and the
discretization.
We choose a simple model problem and
vary the dimension $d$.
We consider $-\Delta u= 1$ in $\Omega:=(0,1)^d$, $u=0$ on $\partial\Omega$
in its variational formulation of finding $u\in H_0^1(\Omega)$
such that $a(u, v):=\int_{\Omega}\innerp{\nabla u(x),\nabla v(x)}dx=1(v)$
for all $v\in H_0^1(\Omega)$.
The corresponding operator is given by $A:H_0^1(\Omega)\rightarrow H^{-1}(\Omega)$,
where $A(u):=a(u,\cdot)$, which is boundedly invertible and self-adjoint.

For the discretization we use tensor products of $L_2$-orthonormal
piecewise polynomial cubic B-spline multiwavelets.
We use our own implementation of an HTucker library. All of the software is
implemented in C++. For more details see, e.g.\ \cite{RuppPhd}.
We set the HT tree to be a perfectly balanced binary tree.
We vary the dimension as $d=2, 4, 8, 16, 32$.

\begin{figure}[ht!]
    \begin{subfigure}{0.47\textwidth}
    \includegraphics[height=42mm]{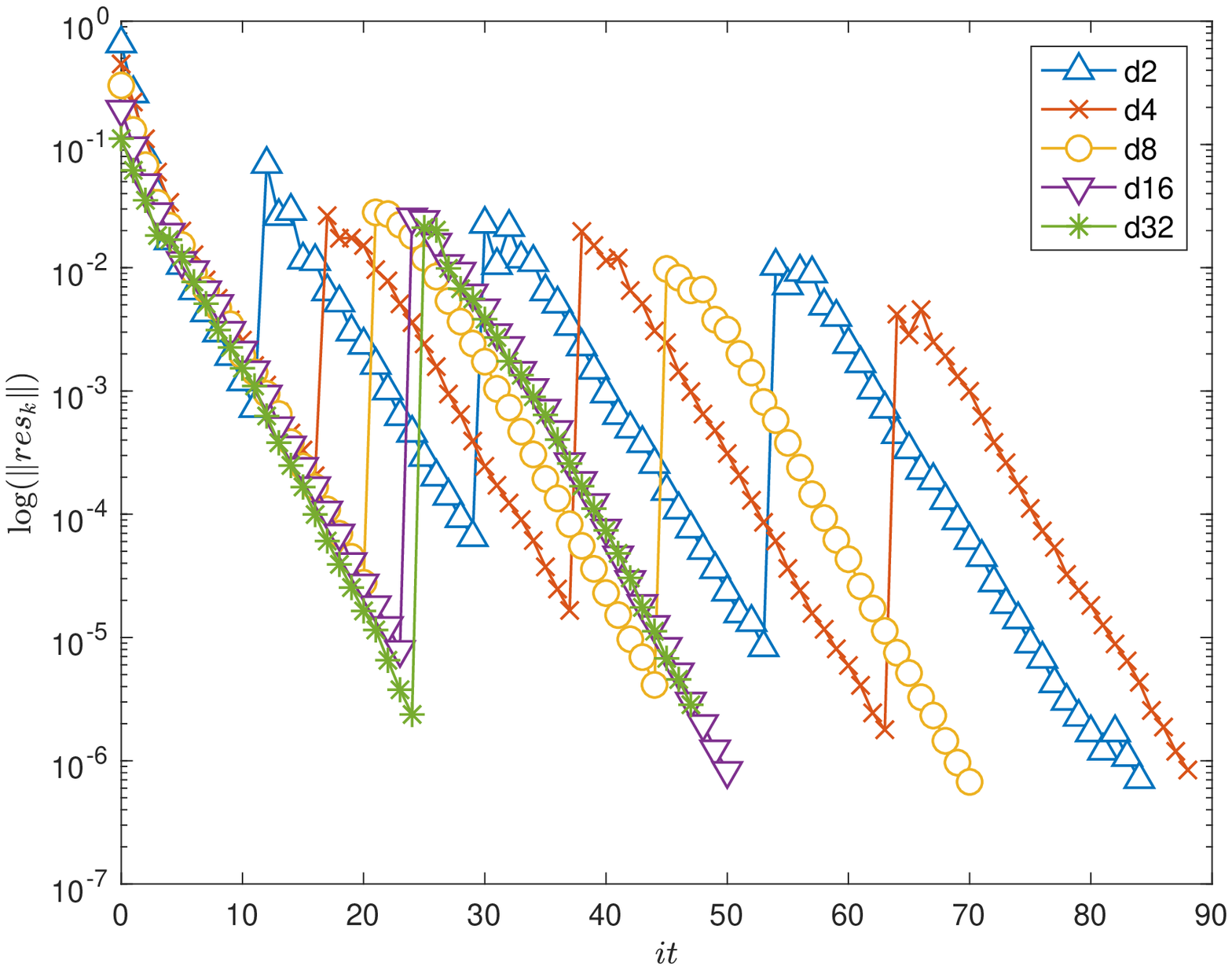}
    \caption{Residual per iteration.\label{fig:its}}
    \end{subfigure}
    \hfill
    \begin{subfigure}{0.47\textwidth}
        \includegraphics[height=42mm]{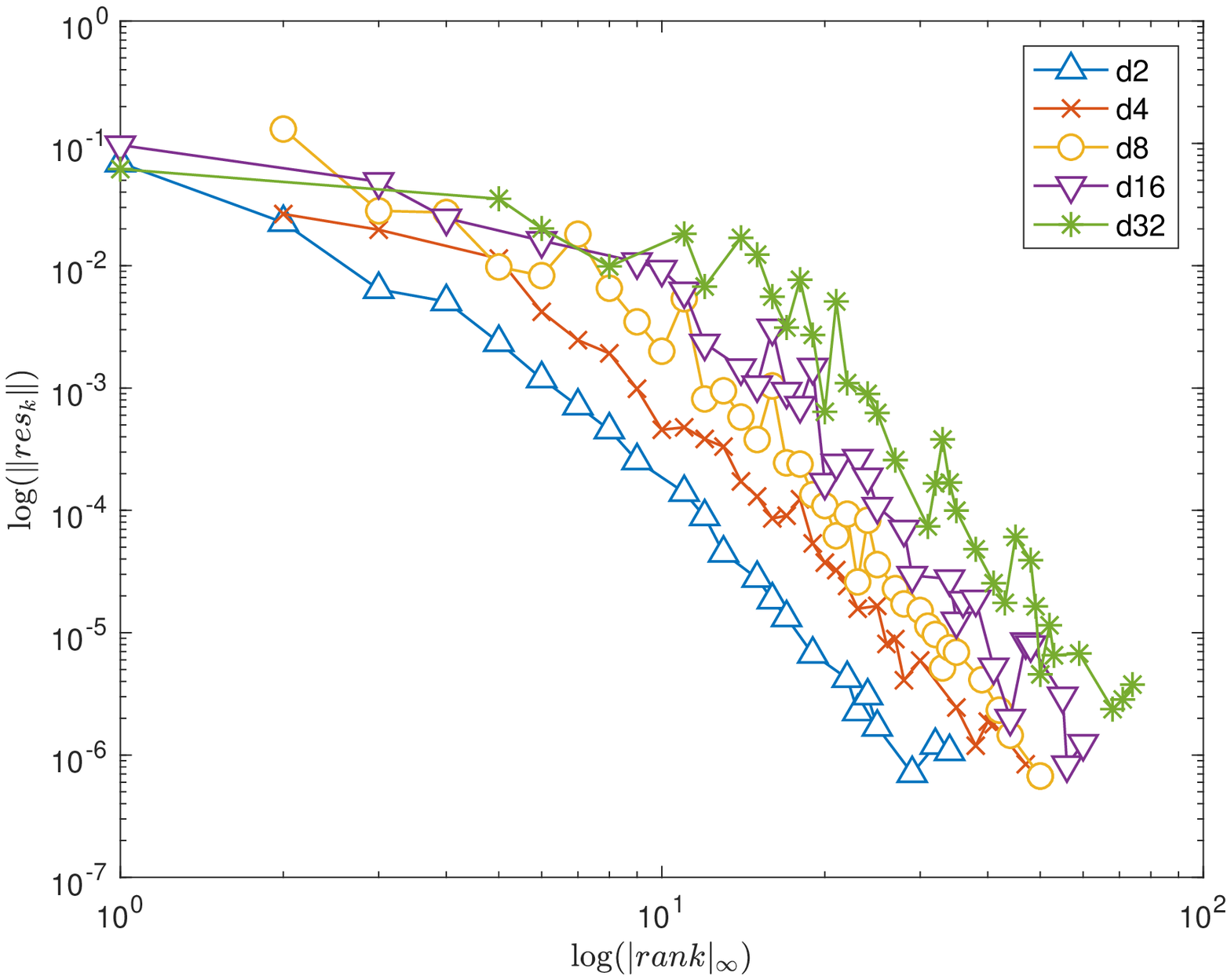}
        \caption{Residual vs.\ ranks.\label{fig:ranks}}
    \end{subfigure}


    \begin{subfigure}{0.47\textwidth}
        \includegraphics[height=42mm]{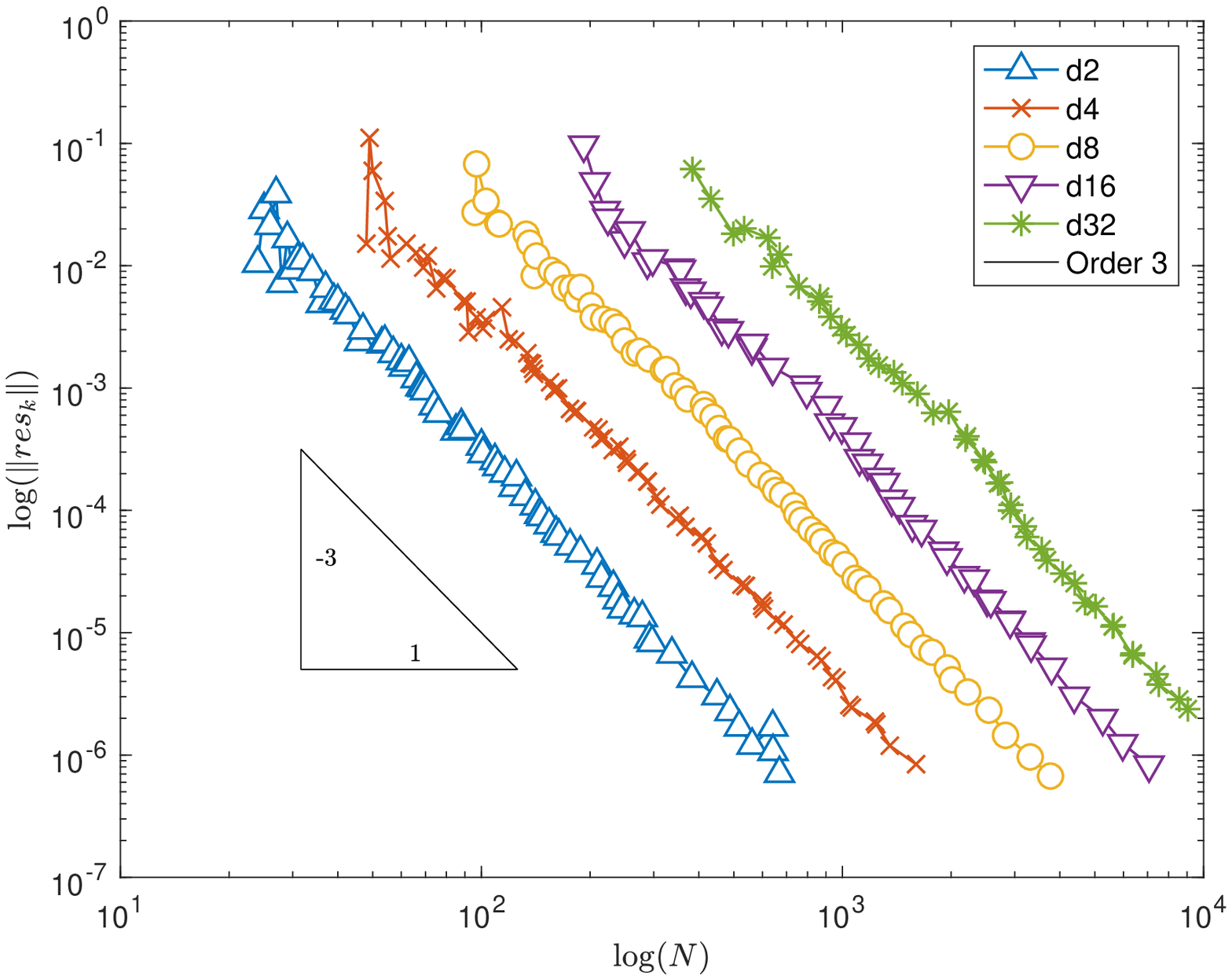}
        \caption{Residual vs.\ size of support of frames.\label{fig:frames}}
    \end{subfigure}
    \hfill
    \begin{subfigure}{0.47\textwidth}
        \includegraphics[height=42mm]{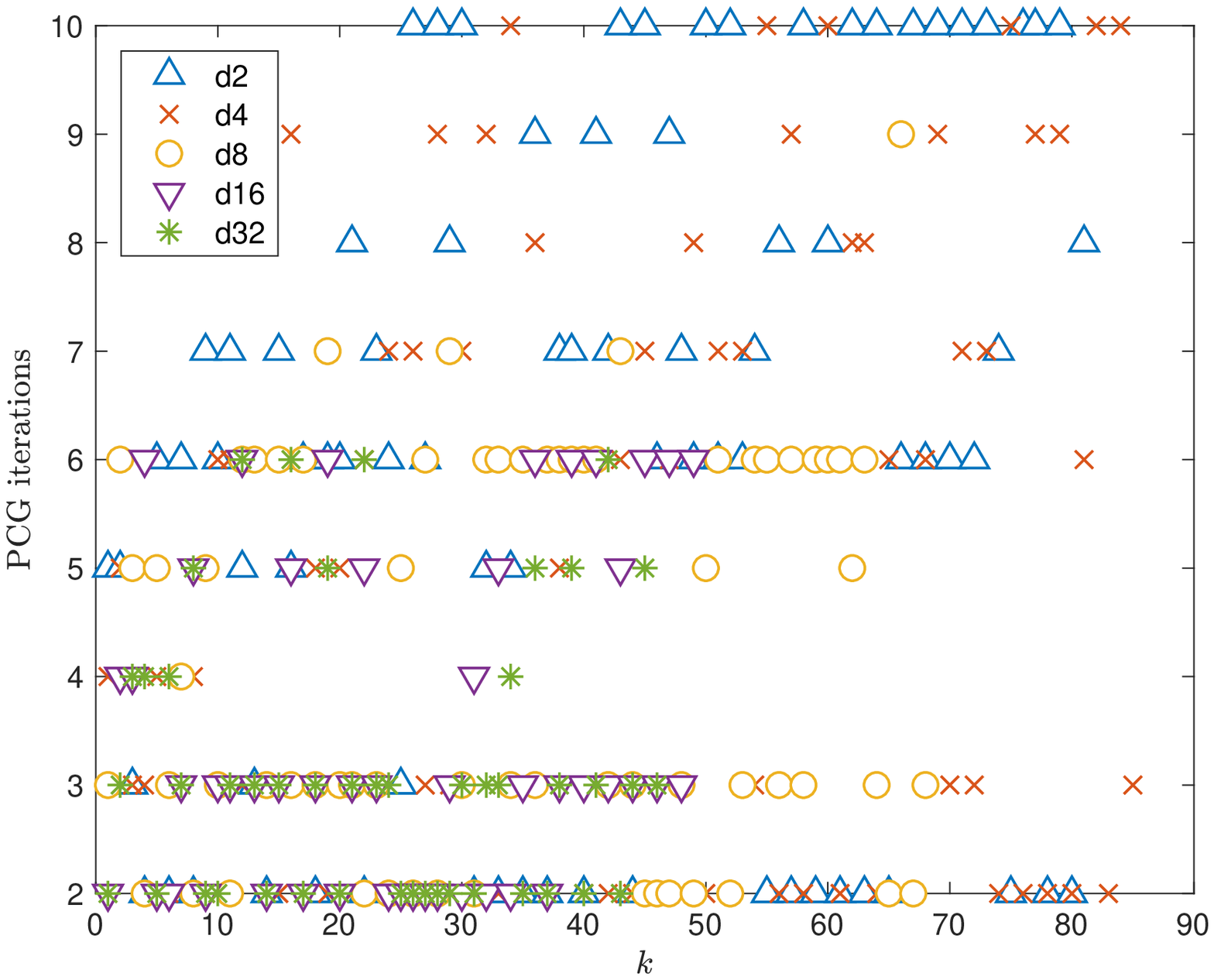}
        \caption{\# \textbf{PCG} iterations per \textbf{HT-AWGM}
                iteration.\label{fig:pcg}}
    \end{subfigure}
    \caption{\textbf{HT-AWGM} for different dimensions $d$.}
    \label{fig:complex}
\end{figure}

\subsubsection*{Results}
In \Cref{fig:its}, we display the convergence history with respect to the
number of overall iterations.
Due to the structure of the linear operator $A$, the
condition number $\kappa(\bs A^\delta)$ is independent of $d$.
Moreover, the parameters $\alpha, \omega_1, \omega_2 \in(0,1)$
are chosen the same for all dimensions.
Thus, the theoretical convergence rate of \textbf{HT-AWGM}
is independent of $d$, which is observed in \Cref{fig:its}.

However, the parameters $\omega_3, \omega_4, \omega_5$ depend on $d$ which
result in different tolerances for the re-truncation and re-coarsening
step\footnote{In the graphics re-truncation and re-coarsening is counted as
one iteration step, though technically it is not a \textbf{HT-AWGM} iteration
step.}.

In \Cref{fig:ranks} we show the behavior of ranks of the numerical solution
$\bs u_k$.
The data points are sorted by rank, where for repeating ranks we took
the minimum of the corresponding residual.
For all dimensions $d$ we observe an exponential decay w.r.t.\ ranks, which
is according to expectation for the Laplacian. As stated in \Cref{rem:complex}
and consistent with the observations in \cite{MarkusSobolev},
we expect the ranks to scale logarithmically in the dimension.

In \Cref{fig:frames} we plot the sum of the supports of frames and the
corresponding residual. Since we are using cubic multi-wavelets, we expect the
convergence w.r.t.\ the support size to be of order 3 and the dimension
dependence to be slightly more than linear.

Finally, \Cref{fig:pcg} shows the number of \textbf{PCG} iterations
in each \textbf{HT-AWGM} iteration. We see that \textbf{PCG} requires
between 2 and 10 iterations to achieve a fixed error reduction
($\omega_2$) for all
dimensions $d$, since $\kappa(\bs A^\delta)$ does not depend on $d$.

\bigskip

We would like to emphasize that, unlike in classical non-tensor adaptive
methods, for high dimensional tensor methods ranks are
crucial for performance. The size of the wavelet discretization
affects the performance indirectly, since the maximum wavelet level affects
the ranks in the preconditioning. However, this is not necessarily a feature
solely of the preconditioning. Larger frames imply we are searching for
low dimensional manifolds in higher dimensional spaces. In the worst case
scenario, this implies the ranks of such manifolds will grow.

A few numerical considerations significantly improve the overall
performance. For \textbf{PCG} choosing the adaptive tolerance is a trade off
between the number iterations and how expensive each
iteration is. We found that choosing the adaptive tolerance $0.1$
yields best results. For experiments varying the adaptive tolerance
we refer to
\cite{ToblerPhd}.

Moreover, note that in each \textbf{PCG} iteration
the preconditioned matrix-vector product only has to be computed once, since
this can be avoided for computing the energy norm of the search
direction.
We are only interested in computing the residual and thus we can also
avoid computing an intermediate matrix-vector product and apply the
preconditioning, summation and truncation to the residual directly. This
gives the same result, but involves much lower intermediate ranks, since
the truncation tolerances are relative to the residual and not
$\|\bs A^\delta\bs u_k\|$.

Finally, in analogy to \cite{DBL2}, if we are interested in controlling
the error only in $L_2$, we can approximate the $L_2$ coefficients.
This means applying $\bs S^{-2}(\delta)\bs A$ instead of
$\bs S^{-1}(\delta)\bs A\bs S^{-1}(\delta)$ which greatly reduces the
computational cost.

\subsubsection*{Adaptivity}
In conclusion we would like to remark on the use of an adaptive discretization
for the model problem above. In a classical AWGM method applied to a smooth
problem like $-\Delta u=1$, we would expect the algorithm to recover a nearly
uniform grid. One might expect the same for the discretization of the frames
in a tensor format. However, as we will see, this is not the case.

\begin{figure}[ht!]
    \begin{subfigure}{0.47\textwidth}
    \includegraphics[height=42mm]{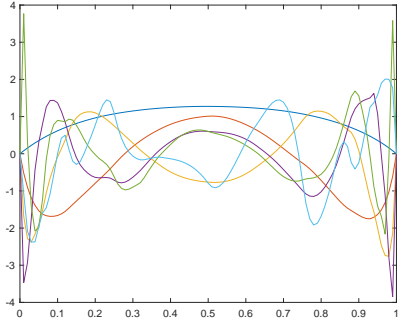}
    \caption{First 6 basis functions.\label{fig:basis}}
    \end{subfigure}
    \hfill
    \begin{subfigure}{0.47\textwidth}
        \includegraphics[height=42mm]{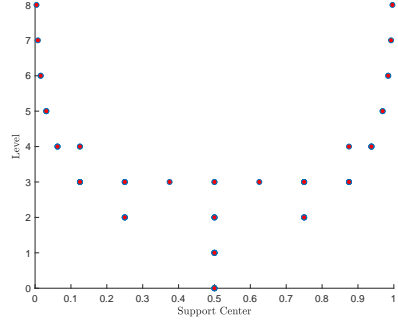}
        \caption{Support centers and levels of wavelets\label{fig:support}}
    \end{subfigure}
    \caption{Basis functions in the first dimension for $d=4$.}
    \label{fig:bfs}
\end{figure}

\Cref{fig:basis} shows the first 6
basis functions in the first dimension of $\bs u_k$ for $d=4$ after
15 inner iterations of \textbf{HT-AWGM}.
\Cref{fig:support}
shows the support centers of active wavelets.

Note that, since we are using cubic multiwavelets, there are more than one
mother scaling functions and mother wavelets. Thus, each point in the plot
can possibly represent more than one wavelet. In this particular case
the overall number of active wavelets is 66 and the maximum level
is 8. The number of wavelets for a uniform grid of up to level 8 is 1536,
which is roughly 23 times more than the number of active wavelets at this stage.
Recall that the computational complexity is linear in the number of active
wavelets.

As we can see, the one
dimensional basis functions exhibit boundary layers and oscillations
for increasing rank. A similar pattern is observed in all dimensions for all
values of $d$.
This behavior can be explained as follows.

Note that the computed one dimensional basis functions do not solve
the original $d$-dimensional equation. Instead, one can consider the best rank
one update. Given a current approximation $u_k$, we compute a rank one update
$v=v_1\otimes\cdots\otimes v_d$ such that
\begin{align*}
    J(u_k+v)=\min_{w=w_1\otimes\cdots\otimes w_d\in H^1_0(\Omega)}J(u_k+w),
\end{align*}
where $J:H^1_0(\Omega)\rightarrow\R$ is the Dirichlet functional
\begin{align*}
    J(u)=\int_\Omega\nabla^2 u\; dx-\int_\Omega fu\; dx,
\end{align*}
for some $f\in L_2(\Omega)$.
Considering the best approximation in the $j$-th dimension and fixing the
rest\footnote{I.e., performing ALS.}, we can compute the first variation as
\begin{align}
    \frac{d}{d\tau}J(u_k+v_1\otimes\cdots (v_j+\tau g)
    \cdots\otimes v_d)\Bigr|_{\tau=0}&=\notag\\
    =\kappa_1\int_{\Omega_j}v_j'g'\;dx_j+
    \kappa_2\int_{\Omega_j}v_jg \;dx_j-\langle R^k_j, g\rangle
    &=0,\quad\forall g\in H_0^1(\Omega_j),\label{eq:als}
\end{align}
with
\begin{align}
    \kappa_1&:=\prod_{i\neq j}\|v_i\|^2_{L_2},\quad
    \kappa_2:=\sum_{k\neq j}\|v_k'\|^2_{L_2}
    \prod_{\substack{i\neq k,\\ i\neq j}}\|v_i\|^2_{L_2},\quad
    \epsilon^{-1}:=\frac{\kappa_2}{\kappa_1}=
    \sum_{k\neq j}\left(\frac{\|v_k'\|_{L_2}}{\|v_k\|_{L_2}}\right)^2,\notag\\
    \langle R^k_j, g\rangle&:=
    \int_{\Omega_j}g\int_{\bigtimes_{i\neq j}\Omega_i}f\cdot
    \otimes_{k\neq j}
    v_k\;dx\notag\\
    &\quad-\int_{\Omega_j}g\int_{\bigtimes_{i\neq j}\Omega_i}
    \nabla^{d-1,\neq j}u_k\cdot\nabla^{d-1,\neq j}
    \otimes_{k\neq j}v_k\; dx\notag\\
    &\quad-\int_{\Omega_j}g'\int_{\bigtimes_{i\neq j}\Omega_i}
    \frac{\partial u_k}{\partial x_j}\cdot
    \otimes_{k\neq j}v_k\; dx.\label{eq:resals}
\end{align}
I.e., the basis functions in \Cref{fig:basis} solve \cref{eq:als}. This has
two consequences. First, this is no longer a Poisson equation, but rather
a reaction-diffusion equation that is singularly perturbed for
$\epsilon\rightarrow 0^+$. Indeed, we have observed that $\epsilon$ becomes
smaller as the rank grows. This explains the boundary layers and the adaptive
discretization in \Cref{fig:bfs}. Second, the right hand side in \cref{eq:als}
is the residual from \cref{eq:resals}. This term has 2 orders of regularity less
than the numerical approximation $u_k$. I.e., using basis functions of higher
regularity improves the regularity of the residual and thus the behavior of
the frames of the numerical approximation. Moreover, the residual also
introduces the oscillations visible in \Cref{fig:basis}.

\section{Conclusion}\label{sec:conclusion}

We introduced and analyzed an adaptive wavelet Galerkin
\\
scheme for high
dimensional elliptic equations.
To deal with the \emph{curse of dimensionality},
we utilized low rank tensor methods, specifically the Hierarchical Tucker
format. The method is adaptive both in the wavelet representation and
the tensor ranks.

We have shown that the method converges and that the numerical
solution has quasi-optimal ranks and wavelet representation.
The computational
complexity depends solely on the Besov regularity and low rank
approximability of the solution. We provided numerical experiments for the
Poisson equation for $d=2, 4, 8, 16$ and $32$ dimensions.

The method is well suited for problems where $d$ is large, provided certain
favorable separability assumptions are satisfied: the operator is
either low rank or can be accurately approximated by low rank operators;
the condition of the operator only mildly depends on the dimension;
the right hand side is either low rank or can be accurately approximated by
low rank functions. The dominating part of the complexity for \textbf{HT-AWGM}
are the ranks. Nonetheless, adaptivity in the wavelet representation pays off
computationally even for smooth problems like the Poisson equation.

\textbf{HT-AWGM} involves a re-truncation and a re-coarsening step to ensure
optimality w.r.t.\ ranks and wavelet representation. Although there are
evidence to support that both can be avoided, we were not able to devise
a rigorous framework to show this. In future work we want to construct a more
clever rank extension strategy, avoid both the
re-truncation and re-coarsening steps and consider different
Galerkin solvers.

\section*{Acknowledgements}

We would like to thank Markus Bachmayr, Rob Stevenson and Wolfgang Dahmen for
their very helpful comments on this work. This paper was partly
written when Mazen Ali was a visiting researcher at Centrale Nantes in
collaboration with Anthony Nouy. We acknowledge Anthony Nouy for
the helpful discussions and financial support. We are grateful to
the European Model Reduction Network (TD COST Action TD1307)
for funding.

\bibliographystyle{acm}
\bibliography{literature}

\end{document}